\documentclass[12pt]{amsart}
\pdfoutput=1
\usepackage{fullpage}
\usepackage{amssymb}
\usepackage{amstext}
\usepackage{amsmath}
\usepackage{amsthm}
\usepackage{stmaryrd}
\usepackage{verbatim}
\usepackage{booktabs}
\usepackage{romannum}
\usepackage[colorlinks,
            linkcolor=blue,
            anchorcolor=blue,
            citecolor=blue
            ]{hyperref}
\usepackage{diagbox}
\usepackage{soul}
\newcommand{\D}{\operatorname{D}}

\newtheorem{theorem}{Theorem}[section]
\newtheorem{corollary}[theorem]{Corollary}
\newtheorem{proposition}[theorem]{Proposition}
\newtheorem{lemma}[theorem]{Lemma}
\theoremstyle{remark}

\newtheorem{example}{Example}

\begin{document}
\title{Evaluation of reciprocal sums of hyperbolic functions using quasimodular forms} 
\author{Wei Wang}
\address{Department of Mathematics, Shaoxing University, Shaoxing 312000, China}
\email{weiwang\_math@163.com}

\begin{abstract}
This paper studies eight families of infinite series involving hyperbolic functions. Under some conditions, these series are linear combinations of derivatives of Eisenstein series. Using complex multiplication theory, the structure of the rings of modular forms and quasimodular forms, and certain differential operators defined on these rings, this paper gives a systematic method for computing the values of these series at CM points.  This paper also expresses the generalized reciprocal sums of Fibonacci numbers as the special values of the series mentioned above. Thus it gives some algebraic independence results about the generalized reciprocal sums of Fibonacci numbers.
\end{abstract}
\thanks{The author is supported by NSFC 11701272 and NSFC 12071221.}
%\today
\maketitle
\pagenumbering{arabic}

\section{Introduction}\label{section_intro}
Let $z$ be a point on the upper half plane and $q=\exp(2\pi i z)$. Let $s\geq 1$ and $p\geq 0$ be two integers. We define the following eight families of series:
\begin{align*}
&\operatorname{\Romannum{1}}_s^p(q)=\sum_{n\geq1}\frac{n^pq^{ns}}{(1-q^{2n})^s},
&\operatorname{\Romannum{5}}_s^p(q)&=\sum_{n\geq1}\frac{(-1)^{n-1}n^pq^{ns}}{(1-q^{2n})^s},\\
&\operatorname{\Romannum{2}}_s^p(q)=\sum_{n\geq 1}\frac{n^pq^{ns}}{(1+q^{2n})^s},
&\operatorname{\Romannum{6}}_s^p(q)&=\sum_{n\geq 1}\frac{(-1)^{n-1}n^pq^{ns}}{(1+q^{2n})^s},\\
&\operatorname{\Romannum{3}}_s^p(q)=\sum_{n\geq1}\frac{(2n-1)^pq^{(n-\frac12)s}}{(1-q^{2n-1})^s},
&\operatorname{\Romannum{7}}_s^p(q)&=\sum_{n\geq1}\frac{(-1)^{n-1}(2n-1)^pq^{(n-\frac12)s}}{(1-q^{2n-1})^s},\\
&\operatorname{\Romannum{4}}_s^p(q)=\sum_{n\geq1}\frac{(2n-1)^pq^{(n-\frac12)s}}{(1+q^{2n-1})^s},
&\operatorname{\Romannum{8}}_s^p(q)&=\sum_{n\geq1}\frac{(-1)^{n-1}(2n-1)^pq^{(n-\frac12)s}}{(1+q^{2n-1})^s}.
\end{align*}
Regarding the above series, we mainly focus on two cases to evaluate their values. The first case is when $z$ is a complex multiplication (CM, for short) point, i.e. an element of an imaginary quadratic field with positive imaginary part. For example, let $z=ci/2$ and $c$ is a square root of a rational number. Then the above series can be expressed as the reciprocal sums of hyperbolic functions:
\begin{align*}
&\operatorname{\Romannum{1}}_s^p(e^{-\pi c})=2^{-s}\sum_{n=1}^{\infty}n^p \operatorname{cosech}^s(n\pi c),
&\operatorname{\Romannum{5}}_s^p(e^{-\pi c})=2^{-s}\sum_{n=1}^{\infty}(-1)^{n-1}n^p \operatorname{cosech}^s(n\pi c),\\
&\operatorname{\Romannum{2}}_s^p(e^{-\pi c})=2^{-s}\sum_{n=1}^{\infty}n^p \operatorname{sech}^s(n\pi c),
&\operatorname{\Romannum{6}}_s^p(e^{-\pi c})=2^{-s}\sum_{n=1}^{\infty}(-1)^{n-1}n^p \operatorname{sech}^s(n\pi c),
\end{align*}
\[
\begin{aligned}
&\operatorname{\Romannum{3}}_s^p(e^{-\pi c})=2^{-s}\sum_{n=1}^{\infty}(2n-1)^p \operatorname{cosech}^s[(2n-1)\pi c/2],\\
&\operatorname{\Romannum{7}}_s^p(e^{-\pi c})=2^{-s}\sum_{n=1}^{\infty}(-1)^{n-1}(2n-1)^p \operatorname{cosech}^s[(2n-1)\pi c/2],\\
&\operatorname{\Romannum{4}}_s^p(e^{-\pi c})=2^{-s}\sum_{n=1}^{\infty}(2n-1)^p \operatorname{sech}^s[(2n-1)\pi c/2],\\
&\operatorname{\Romannum{8}}_s^p(e^{-\pi c})=2^{-s}\sum_{n=1}^{\infty}(-1)^{n-1}(2n-1)^p \operatorname{sech}^s[(2n-1)\pi c/2].
\end{aligned}
\]

These series have attracted the interest of many authors. In his famous letter to Hardy in 1913, Ramanujan gave an expression 
\[
\text{I}^{4n}_2(e^{-\pi})=\frac{n}{\pi}\left(\frac{B_{4n}}{8n}+\frac{1^{4n-1}}{e^{2\pi}-1}+\frac{2^{4n-1}}{e^{4\pi}-1}+\frac{3^{4n-1}}{e^{6\pi}-1}+\cdots\right).
\]
Further examples can be found in his notebooks \cite{MR2952081,MR1486573,MR0099904}. For example, in \cite[Chapter 17]{MR2952081}, a lot of the series are evaluated in terms of the complete elliptic integrals. In a series of papers \cite{MR0348150,MR0364692,MR0467067}, Ling studied some of the above series for $p=0$ and gave their exact values for $c=1,\sqrt{3},1/\sqrt{3}$. Zucker \cite{MR0516762} expressed a lot of the above series mostly for $p=0$ as the Jacobian elliptic functions. He showed that when a certain parameter in these series is the square root of a rational number the series can be expressed in terms of $\Gamma$-function and $\pi$.  
Xu and Zhao \cite{xu2022functional} used residue theorem and asymptotic formulas of trigonometric and hyperbolic functions to express a lot of families of the above series for $p=2$ in terms of some special values of $\Gamma$-function and $\pi$. They also proved that for higher $p$, the series can be evaluated by recurrence formulas. 

Berndt \cite{MR3574639} showed that certain classes of infinite integrals containing trigonometric and hyperbolic trigonometric functions can be expressed by the reciprocal sums of  hyperbolic functions. For example, let $a$ denote a non-negative integer, Berndt proved that
\[
(1+i^{2a})\int_{0}^{\infty}\frac{x^{2a+1}dx}{\cos x+\cosh x}=\frac{\pi^{2a+2}i^{a}}{2^{a}}\sum_{n=1}^{\infty}(-1)^{n-1}(2n-1)^{2a+1}\operatorname{sech}[(2n-1)\pi/2]
\]
and
\[
(1+i^{a+2})\int_{0}^{\infty}\frac{x^{a+1}dx}{\cos x-\cosh x}=2i(1+i)^{a}\pi^{a+2}\sum_{n=1}^{\infty}(-1)^{n-1}n^{a+1}\operatorname{cosech}(n\pi).
\]
Rui, Xu, Yang and Zhao \cite{RUI2025128676} considered more general classes of Berndt-type integrals, and they evaluated these integrals in closed forms by expressing these integrals as reciprocal sums of hyperbolic functions. 

The methods used in this paper are distinct from those mentioned in the previous literature. We apply the theory of quasimodular forms. As a result, we obtain a more systematic way to evaluate these series. In particular, we provide a complete list of those series that can be represented by derivatives of Eisenstein series (Theorem \ref{theorem_main}). 
By the theory of complex multiplication, values of the series listed in Theorem \ref{theorem_main} at CM points can be expressed using special values of $\Gamma$-function and $\pi$.

We also consider these series from another perspective, focusing on the case where $q$ is an algebraic number. Recall that the sequence of Fibonacci numbers $F_n$ is defined by the following formula:
\[
F_n=F_{n-1}+F_{n-2},~F_0=0,~F_1=1,~n\geq 0.
\]
It is well-known that they have the explicit expressions
\[
F_n=\frac{\alpha^n-\beta^n}{\alpha-\beta}\text{, where }\alpha=\frac{1+\sqrt{5}}{2},\beta=-\alpha^{-1}.
\]
Expressing the reciprocal sums of Fibonacci numbers as special values of certain functions has a long history. Landau \cite{MR1504354} showed that
\[
\frac{1}{\sqrt{5}}\sum_{n=0}^{\infty}\frac1{F_{2n}}=L\left(\frac{3-\sqrt{5}}{2}\right)-L\left(\frac{7-3\sqrt{5}}{2}\right).
\]
and
\[
\sum_{n=0}^{\infty}\frac1{F_{2n+1}}=\frac{\sqrt{5}}{4}\theta_2^2\left(\frac{3-\sqrt{5}}{2}\right),
\]
where $L(q)=\sum_{n\geq 1}\frac{q^n}{1-q^n}$ and $\theta_2(q)=1+2\sum_{n\geq 1}q^{n^2/4}$. In \cite{MR1628840}, more examples are presented, for example 
\[
\sum_{n=1}^{\infty}\frac{(-1)^n}{F_{2n}^2}=\frac{(\alpha-\beta)^2}{24}E_2(\beta^2),
\]
where $E_2(q)=1-24\sum_{n=1}^{\infty}\frac{nq^n}{1-q^n}$.

More generally, one can define the Fibonacci zeta function by the series
\[
\zeta_{F}(s):=\sum_{n=1}^{\infty}F_n^{-s}.
\]
The series converges for $\mathfrak{R}(s)> 0$ due to the exponential growth of Fibonacci numbers. Readers interested in the Fibonacci zeta function can read the survey by R. Murty \cite{murty2013fibonacci}. Using the  the explicit expression for Fibonacci numbers, we get
\begin{align*}
\zeta_{F}(s)&=5^{s/2}\sum_{n\geq 1}\frac{1}{((-1)^n\beta^{-n}-\beta^{n})^s}=5^{s/2}\left(\sum_{n\geq 1}\frac{\beta^{2ns}}{(1-\beta^{4n})^s}+(-1)^s\sum_{n\geq 1}\frac{\beta^{(2n-1)s}}{(1-\beta^{(4n-2)})^s}\right)\\
&=5^{s/2}\left(\text{I}_s^0(\beta^2)+(-1)^s\text{IV}_s^0(\beta^2)\right).
\end{align*}
Inspired by the above expression, we define the generalized Fibonacci zeta function by the series
\[
\zeta^p_{F}(s):=\sum_{n=1}^{\infty}n^pF_n^{-s}.
\]
The special values of this function and its various extensions can be expressed as some linear combinations of the series introduced at the beginning of the paper.

We are interested in special values of the Fibonacci zeta function at positive integers, especially concerning their transcendence. Using Nesterenko's theorem on the Ramanujan functions, Elsner, Shimomura  and Shiokawa  proved the algebraic independence results concerning the special values of the Fibonacci zeta function in a series of papers \cite{MR2354148,MR2456843,MR2794928}. They proved that the three numbers $\zeta_F(2s_1),\zeta_F(2s_2),\zeta_F(2s_3)$ are algebraically independent if and only if one of the three distinct positive integers $s_1,s_2,s_3$ are even. Elsner, Shimomura  and Shiokawa's approach is to express the values of the Fibonacci zeta function at positive even integers in terms of elliptic integrals explicitly. Thus, by utilizing Nesterenko's theorem in conjunction with a criterion for algebraic independence, they obtained their theorem. In this paper, we discuss the algebraic independence of the special values of the generalized Fibonacci zeta function. We can see that without knowing the exact values of the generalized Fibonacci zeta function, we can establish their algebraic independence by certain properties of quasimodular forms.

The paper is organized as follows. Section \ref{section_pre} introduces some notation and fundamental concepts about modular forms and quasimodular forms. In Section \ref{section_derandser}, we express the series defined at the beginning as linear combinations of quasimodular forms, serving as the foundation for subsequent discussions. Section \ref{section_cmvalue} presents some methods for computing the values of modular forms and quasimodular forms at CM points, and we provide detailed calculations for two examples. In Section \ref{section_fibonacci}, we express special values of the generalized Fibonacci zeta function as values of quasimodular forms. Additionally, we discuss some of their algebraic independence results.

\section{Preliminaries}\label{section_pre}
Let $\chi$ be a Dirichlet character. For any matrix $\gamma=\begin{pmatrix}
a&b\\
c&d
\end{pmatrix}\in \operatorname{M}_2(\mathbb{Z})$, define $\chi(\gamma):=\overline{\chi(a)}$. Let $k$ be a positive integer. For function $f$ defined on the upper half plane,  the slash operator of weight $k$ is defined by
\[
(f|_k\gamma)(z)=(ad-bc)^{k/2}(cz+d)^{-k}f\left(\frac{az+b}{cz+d}\right).
\]
Now let $\chi$ be a character modulo $N$ such that $\chi(-1)=(-1)^k$. Recall that a modular form in $M_k(\Gamma_0(N),\chi)$ is a holomorphic function on the upper half plane which satisfies
\begin{equation}\label{trans_mod}
(f|_k\gamma)(z)=\chi(\gamma)f(z), \text{ for every }\gamma=\begin{pmatrix}
a&b\\
c&d
\end{pmatrix} \in\Gamma_0(N),
\end{equation}
and also is holomorphic at cusps. For $f\in M_k(\Gamma_0(N),\chi)$ and positive integer $n$, the $V_n$ operator is defined by $\left(f|_{V_n}\right)(z):=f(nz)$. It can be verified directly that $f|_{V_n}\in M_k(\Gamma_0(nN),\chi)$.

A typical example of modular forms for $\operatorname{SL}_2(\mathbb{Z})$ is the Eisenstein series of weight $2k\geq 4$, which is defined by the series
\[
E_{2k}(z)=\frac{1}{2}\sum_{\substack{(c,d)\in\mathbb{Z}^2 \\ \gcd(c,d)=1}}
\frac{1}{(cz+d)^{2k}}.
\]
It has Fourier expansion
\[
E_{2k}(z)=1-\frac{4k}{B_{2k}}\sum_{n\geq 1}\sigma_{2k-1}(n)q^n=1-\frac{4k}{B_{2k}}\sum_{n,m\geq 1}n^{2k-1}q^{nm},~q=e^{2\pi iz},
\]
with the Bernoulli numbers $B_{k}$ defined by
\[
\frac{t}{e^t-1}:=\sum_{k=0}^{\infty}B_k\frac{t^k}{k!}.
\]
For primitive Dirichlet characters $\psi,\varphi$ such that $(\psi\varphi)(-1)=(-1)^k$, we define the generalized divisor sum as
\[
\sigma_{k-1}^{\psi,\varphi}(n)=\sum_{m\mid n}\psi(n/m)\varphi(m)m^{k-1}.
\]
For the Dirichlet character $\psi$ modulo $f$, the generalized Bernoulli numbers  $B_{k,\psi}$ are defined by 
\[
\sum_{u=0}^{f-1}\psi(t)\frac{te^{ut}}{e^{ft}-1}=\sum_{k=0}^{\infty}B_{k,\psi}\frac{t^k}{k!}.
\]
If $k\geq 3$, define the $q$-series 
\[
E_k^{\psi,\varphi}(z)=-\delta(\psi)\frac{B_{k,\psi}}{2k}+\sum_{n=1}^{\infty}\sigma_{k-1}^{\psi,\varphi}(n)q^n,
\]
where $\delta(\varphi)=1$ if $\varphi=1$, and 0 otherwise.
For $k=1$, define the $q$-series 
\[
E_1^{\psi,\varphi}(z)=-\delta(\varphi)\frac{B_{1,\psi}}2-\delta(\psi)\frac{B_{1,\phi}}{2}+\sum_{n=1}^{\infty}\sigma_{0}^{\psi,\varphi}(n)q^n.
\]
The theory of Eisenstein series tells us that $E_k^{\psi,\varphi}(z)\in M_k(\Gamma_0(N),\chi)$, where $\chi=\psi\varphi$ and $N$ is the modulus of $\chi$, cf. \cite{MR2112196}.
We also recall that the Eisenstein series of weight 2 is defined by
\[
E_2(z)=1-24\sum_{n=1}^{\infty}\sigma_1(n)q^n.
\]
It is well-known that the function $E_2(z)$ has the  transformation property
\begin{equation}\label{trans_E2}
(cz+d)^{-2}E_2\left(\frac{az+b}{cz+d}\right)=E_2(z)+\frac{6}{\pi i}\frac{c}{cz+d}, \text{ for every }\begin{pmatrix}
a&b\\
c&d
\end{pmatrix} \in\operatorname{SL}_2(\mathbb{Z}).
\end{equation}
By the above transformation property, we can see that $NE_2(Nz)-E_2(z)$ is a modular form in $M_2(\Gamma_0(N))$. 

Let $f$ be a modular form in $M_k(\Gamma_0(N),\chi)$ with Fourier expansion $\sum_{n\geq 0}a_n q^n$. The differential operator is defined by
\[
\D f=\frac{1}{2\pi i}\frac{df}{dz}=q\frac{df}{dq}=\sum_{n=1}^{\infty}na_nq^n.
\]
Taking the derivative of both sides of the equation (\ref{trans_mod}), we see that the differential operator $\D$ satisfies
\begin{equation}\label{trans_Df}
(\D f|_{k+2}\gamma)(z)=\chi(\gamma)\D f(z)+\chi(\gamma)\frac{k}{2\pi i}\frac{c}{cz+d}f(z).
\end{equation}
Hence $\D f$ is not a modular form but a quasimodular form. Recall that a quasimodular form  of weight $k$ is a function satisfying the transformation property
\[
\left(f|_k\gamma\right)(z)=\chi(\gamma)\sum_{0\leq i\leq p} F_i(z)\left(\frac{c}{cz+d}\right)^i \text{ for every }\gamma=\begin{pmatrix}
a&b\\
c&d
\end{pmatrix} \in\Gamma_0(N),
\]
where each $F_i(z)$ is a holomorphic function defined on the upper half plane and is holomorphic at cusps. If $F_p(z)\neq 0$, we say that $f(z)$ is of depth $p$.  By repeatedly taking derivatives of both sides of the equation (\ref{trans_Df}), we see that if $f$ is a modular form of weight $k$, then $\D^p f$ is a quasimodular form of weight $k+2p$ and depth $p$. By (\ref{trans_E2}), the function $E_2$ offers a typical example of quasimodular forms of weight $2$ and depth 1. Indeed, every quasimodular form is a polynomial in $E_2$ with modular coefficients (for more detail see \cite{MR1363056}).

\section{expressing series as derivatives of Eisenstein series}\label{section_derandser}
Following Zucker's notation in \cite{MR0516762}, we define the following sixteen families of $q$-series:
\[
\begin{aligned}
A_s(q)&=\sum_{n=1}^{\infty}\frac{n^sq^n}{1-q^n}, 
&B_s(q)&=\sum_{n=1}^{\infty}\frac{(-1)^{n-1}n^sq^n}{1-q^n},\\
C_s(q)&=\sum_{n=1}^{\infty}\frac{n^sq^n}{1-q^{2n}},
&D_s(q)&=\sum_{n=1}^{\infty}\frac{(-1)^{n-1}n^sq^n}{1-q^n}, \\
F_s(q)&=\sum_{n=1}^{\infty}\frac{(2n-1)^sq^{2n-1}}{1-q^{2n-1}},
&G_s(q)&=\sum_{n=1}^{\infty}\frac{(-1)^{n-1}(2n-1)^sq^{2n-1}}{1-q^{2n-1}},\\
H_s(q)&=\sum_{n=1}^{\infty}\frac{(2n-1)^sq^{n-\frac12}}{1-q^{2n-1}},
&J_s(q)&=\sum_{n=1}^{\infty}\frac{(-1)^{n-1}(2n-1)^sq^{n-\frac12}}{1-q^{2n-1}},\\
L_s(q)&=\sum_{n=1}^{\infty}\frac{n^sq^{n}}{1+q^{n}},
&M_s(q)&=\sum_{n=1}^{\infty}\frac{(-1)^{n-1}n^sq^{n}}{1+q^{n}},\\
N_s(q)&=\sum_{n=1}^{\infty}\frac{n^sq^{n}}{1+q^{2n}},
&P_s(q)&=\sum_{n=1}^{\infty}\frac{(-1)^{n-1}n^sq^{n}}{1+q^{2n}},\\
Q_s(q)&=\sum_{n=1}^{\infty}\frac{(2n-1)^sq^{2n-1}}{1+q^{2n-1}},
&R_s(q)&=\sum_{n=1}^{\infty}\frac{(-1)^{n-1}(2n-1)^sq^{2n-1}}{1+q^{2n-1}},\\
T_s(q)&=\sum_{n=1}^{\infty}\frac{(2n-1)^sq^{n-\frac12}}{1+q^{2n-1}},
&U_s(q)&=\sum_{n=1}^{\infty}\frac{(-1)^{n-1}(2n-1)^sq^{n-\frac12}}{1+q^{2n-1}}.
\end{aligned}
\]
These $q$-series can be divided into two classes. The first class is 
\[
\left\{A_s(q),B_s(q),C_s(q),D_s(q),F_s(q),H_s(q),L_s(q),M_s(q),Q_s(q),U_s(q)\right\},
\]
and the second class is 
\[
\left\{G_s(q),J_s(q),N_s(q),P_s(q),R_s(q),T_s(q)\right\}.
\]

From now on, $\chi_{-4}$ denotes the Dirichlet character defined by
\[
\chi_{-4}(d):=\left(\frac{-4}{d}\right)=
\left\{
\begin{aligned}
&1 &d\equiv 1 \mod 4,\\
&-1& d\equiv 3 \mod 4,\\
&0&\text{otherwise.}
\end{aligned}\right.
\]
For $q$-series $f=\sum_{n\geq 0}a_nq^n$, the twist $f\otimes \chi_{-4}$ is defined by 
\[
f\otimes \chi_{-4}=\sum_{n\geq 0}\chi_{-4}(n)a_nq^n.
\]
The following proposition provides the basis for classification.
\begin{proposition}\label{prop_expressqseries}
The $q$-series in the first class can be expressed in terms of $A_s(q)$:
\[
\begin{aligned}
B_s(q)&=A_s(q)-2^{s+1}A_s(q^2),
&C_s(q)&=A_s(q)-A_s(q^2),\\
D_s(q)&=2^{s+1}A_s(q^4)-(2^{s+1}+1)A_s(q^2)+A_s(q),
&F_s(q)&=A_s(q)-2^sA_s(q^2),\\
H_s(q^2)&=2^{s}A_s(q^4)-(2^{s}+1)A_s(q^2)+A_s(q),
&L_s(q)&=A_s(q)-2A_s(q^2),\\
\end{aligned}
\]
\[
\begin{aligned}
M_s(q)&=2^{s+2}A_s(q^4)-(2^{s+1}+2)A_s(q^2)+A_s(q),\\
Q_s(q)&=2^{s+1}A_s(q^4)-(2^s+2)A_s(q^2)+A_s(q),\quad
U_s(q^2)=(A_s\otimes \chi_{-4})(q).
\end{aligned}
\]
The $q$-series in the second class can be expressed in terms of $G_s(q)$ and $N_s(q)$: 
\[
\begin{aligned}
J_s(q^2)&=G_s(q)-G_s(q^2),
&P_s(q)&=N_s(q)-2^{s+1}N_s(q^2),\\
R_s(q)&=G_s(q)-2G_s(q^2),
&T_s(q^2)&=N_s(q)-2^sN_s(q^2).
\end{aligned}
\]
\end{proposition}
\begin{proof}
The proof is quite straightforward. Note that 
\[
A_s(q)=\sum_{n=1}^{\infty}\frac{n^sq^n}{1-q^n}=\sum_{n\geq1}\sigma_s(n)q^n,~\sigma_s(n)=\sum_{d\mid n}d^s.
\]
The relation $B_s(q)=A_s(q)-2^{s+1}A_s(q^2)$ is equivalent to an identity of divisor functions: $\sum_{d\mid n}(-1)^{d-1}d^s=\sigma_s(n)-2^{s+1}\sigma_s(n/2)$. To see it, it suffices to note that
\[
\sum_{d\mid n}(-1)^{d-1}d^s=\sum_{\substack{d\mid n\\d\text{ odd}}}d^s-\sum_{\substack{d\mid n\\d\text{ even}}}d^s,~\sigma_s(n)=\sum_{\substack{d\mid n\\d\text{ odd}}}d^s+\sum_{\substack{d\mid n\\d\text{ even}}}d^s
\]
and
\[
\sum_{\substack{d\mid n\\d\text{ even}}}d^s=\sum_{\substack{2d_1\mid 2n_1\\n=2n_1}}d^s=\sum_{\substack{d_1\mid n_1\\n=2n_1}}(2d_1)^{s}=2^s\sigma_s(n/2) \text{, so }\sum_{\substack{d\mid n\\d\text{ odd}}}d^s=\sigma_s(n)-2^s\sigma_s(n/2).
\]
The above calculation also demonstrate the identity $F_s(q)=A_s(q)-2^sA_s(q^2)$. 

The two identities
\[
\frac{n^sq^n}{1-q^{2n}}=\frac{n^sq^n}{1-q^n}-\frac{n^sq^{2n}}{1-q^{2n}}\text{ and }\frac{(2n-1)^sq^{2n-1}}{1-q^{4n-2}}=\frac{(2n-1)^sq^{2n-1}}{1-q^{2n-1}}-\frac{(2n-1)^sq^{4n-2}}{1-q^{4n-2}}
\]
imply that $C_s(q)=A_s(q)-A_s(q^2)$ and $H_s(q^2)=F_s(q)-F_s(q^2)$.  And we have $D_s(q^2)=H_s(q^2)-2^sC(q^2)$, hence $D_s(q)=2^{s+1}A_s(q^4)-(2^{s+1}+1)A_s(q^2)+A_s(q)$ by the above discussions. The relation $L_s(q)=C_s(q)-A_s(q^2)=A_s(q)-2A_s(q^2)$ follows by the identity
\[
\frac{n^sq^n}{1+q^n}=\frac{n^sq^n}{1-q^{2n}}-\frac{n^sq^{2n}}{1-q^{2n}}.
\] 
By
\[
\frac{(2n-1)^sq^{2n-1}}{1+q^{2n-1}}=\frac{(2n-1)^sq^{2n-1}}{1-q^{4n-2}}-\frac{(2n-1)^sq^{4n-2}}{1-q^{4n-2}},
\]
we obtain $Q_s(q)=H_s(q^2)-F_s(q^2)=2^{s+1}A_s(q^4)-(2^s+2)A_s(q^2)+A_s(q)$. And  clearly $M_s(q)=2^sL_s(q^2)-Q_s(q)=2^{s+2}A_s(q^4)-(2^{s+1}+2)A_s(q^2)+A_s(q)$. Finally, by a simple calculation,
\[
\begin{aligned}
U_s(q^2)&=\sum_{n=1}^{\infty}\frac{(-1)^{n-1}(2n-1)^sq^{2n-1}}{1+q^{4n-2}}=\sum_{n,m\geq 1}(-1)^{n+m}(2n-1)^sq^{(2n-1)(2m-1)}\\
&=\sum_{\substack{n\geq1\\n\text{ odd}}}(-1)^{\frac{1-n}{2}}\left(\sum_{d\mid n}d^s\right)q^n=\sum_{n\geq 1}\chi_{-4}(n)\sigma_s(n)q^n
\end{aligned}
\]
we deduce that $U_s(q^2)=(A_s\otimes \chi_{-4})(q)$.

The relationships of the series in the second class follow by a similar pattern as before, we omit it.
\end{proof}

By the discussion in Section \ref{section_pre}, $A_{2s-1}(q)$ is the $q$-series of the Eisenstein series of weight $2s$. More precisely,
\[
A_{2s-1}(q)=\frac{B_{2s}}{4s}(1-E_{2s}(z)), q=e^{2\pi iz}.
\]
Recall that if $f\in M_{k}(\operatorname{SL}_2(\mathbb{Z}))$, then $f\otimes \chi_{-4}:=\frac{1}{2i}\left(f(z+\frac{1}{4})-f(z-\frac{1}{4})\right)\in M_k(\Gamma_0(16))$. Hence by Proposition \ref{prop_expressqseries}, when $s\geq 3$ is an odd positive integer, the series in the first class are $q$-series of modular forms (up to some constants). When $s=1$, by the transformation property of $E_2$, the four $q$-series $F_1,L_1,Q_1,U_1$ are $q$-series of modular forms and the rest are $q$-series of quasimodular forms.

For the series in the second class, we have 
\[
E_1^{1,\chi_{-4}}(z)=E_1^{\chi_{-4},1}(z)=\frac{1}{4}+G_0(q)=\frac14+N_0(q),
\]
and for $s\geq 1$,
\[
G_{2s}(q)=E_{2s+1}^{1,\chi_{-4}}(z)+\frac{B_{2s+1,\chi_{-4}}}{4s+2},\quad N_{2s}(q)=E_{2s+1}^{\chi_{-4},1}(z).
\]
By Proposition \ref{prop_expressqseries}, when $s$ is an even non-negative integer, the $q$-series in the second class are modular forms. In the subsequent discussions, when discussing the series in the first class, we focus only on the case where $s$ is odd and when discussing the series in the second class, we focus only on the case where $s$ is even.

The series defined in Section \ref{section_intro} can be expressed by the above series  or their derivatives. To begin with, we introduce a sequence called the central factorial numbers, defined as follows: 
\[
\begin{aligned}
&x^2\prod_{k=1}^{n-1}(x^2-k^2)=\sum_{k=1}^nt(2n,2k)x^{2k},\\
&x\prod_{k=1}^n (x^2-\frac{(2k-1)^2}{4})=\sum_{k=0}^nt(2n+1,2k+1)x^{2k+1}.
\end{aligned}
\]
These numbers are indeed the same as the numbers $\alpha_m(s)$ and $\beta_m(s)$ defined by Zucker in \cite{MR0516762}. They are related by
\[
\begin{aligned}
\alpha_m(s)=t(2s,2s-2m),~\beta_m(s)=4^mt(2s+1,2s-2m+1).
\end{aligned}
\]
For a deeper discussion of the central factorial numbers, we refer the readers to the good survey \cite{MR1002886}. For $q$-series $\sum_{\alpha\in\frac{1}{2}+\mathbb{Z}}a_{\alpha}q^{\alpha}$, the $\D$ operator is defined by
\[
\D \left(\sum_{\alpha\in\frac{1}{2}+\mathbb{Z}}a_{\alpha}q^{\alpha}\right):= \sum_{\alpha\in\frac{1}{2}+\mathbb{Z}}\alpha a_{\alpha}q^{\alpha}.
\]

Using the identity
\[
\frac{1}{(1-q)^k}=\sum_{m\geq 0}\binom{k+m-1}m q^m,
\]
we get
\[
\begin{aligned}
\text{I}^{2p}_{2s}(q)&=\sum_{n\geq 1}\frac{n^{2p}q^{2ns}}{(1-q^{2n})^{2s}}=\frac{1}{(2s-1)!}\sum_{m\geq 0,n\geq 1}\frac{(m+2s-1)!}{m!}n^{2p}q^{2n(m+s)}\\
&=\frac{1}{(2s-1)!}\sum_{m\geq 0,n\geq 1}(m+s)((m+s)^2-(s-1)^2)\cdots((m+s)^2-1)n^{2p}q^{2n(m+s)}\\
&=\frac{1}{(2s-1)!}\sum_{r=1}^{s}t(2s,2r)\sum_{m\geq 0,n\geq 1}(m+s)^{2r-1}n^{2p}q^{2n(m+s)}.
\end{aligned}
\]
By the definition of $t(2s,2r)$, for $0\leq m\leq s-1$, $\sum_{r=0}^{s-1}t(2s,2r)m^{2r-1}=0$, hence the above equals to
\[
\frac{1}{(2s-1)!}\sum_{r=1}^{s}t(2s,2r)\sum_{n,m\geq 1}m^{2r-1}n^{2p}q^{2nm}.
\]
By $\sum_{n,m\geq 1}n^am^bq^{nm}=\D ^{\min\{a,b\}}A_{|a-b|}(q)$,
we get
\[
\text{I}_{2s}^{2p}(q)=\frac{1}{(2s-1)!}\sum_{r=1}^st(2s,2r)\D^{\min\{2p,2r-1\}}A_{|2r-2p-1|}(q^2).
\]
Similarly, we have
\[
\begin{aligned}
\text{I}_{2s+1}^{2p+1}(q)&=\frac{1}{(2s)!}\sum_{m\geq 0,n\geq 1}\frac{(m+2s)!}{m!}n^{2p+1}q^{(2m+2s+1)n}\\
&=\frac{1}{(2s)!}\sum_{m\geq 0,n\geq 1}((m+s+\frac12)^2-(s-\frac12)^2)\cdots ((m+s+\frac12)^2-\frac 14)n^{2p+1}q^{(2m+2s+1)n}\\
&=\frac{1}{(2s)!}\sum_{r=0}^st(2s+1,2r+1)\sum_{n,m\geq 1}2^{-2r}(2m-1)^{2r}n^{2p+1}q^{(2m-1)n}.
\end{aligned}
\]
By
\[
\sum_{m,n\geq 1}(2m-1)^an^bq^{(2m-1)n}=\left\{
\begin{aligned}
~\D^aC_{b-a}(q) \qquad\text{ if }&a<b,\\
~\D^bF_{a-b}(q)  \qquad\text{ if }&a>b,
\end{aligned}\right.
\]
we deduce that
\[
\begin{aligned}
(2s)!\text{I}_{2s+1}^{2p+1}(q)=&\sum_{r=0}^{p}2^{-2r}t(2s+1,2r+1)\D^{2r}C_{2p-2r+1}(q)\\
&+\sum_{r=p+1}^{s}2^{-2r}t(2s+1,2r+1)\D^{2p+1} F_{2r-2p+1}(q).
\end{aligned}
\]
Here, if the summation index exceeds the range, it implies that no summation is performed. Similarly to the above operation, we can express the series in Section \ref{section_intro} as finite sums of the derivatives of the $q$-series defined at the beginning of this section. We list them in the following.
\begin{theorem}\label{theorem_main}
The series defined in Section \ref{section_intro} can be expressed by derivatives of Eisenstein series if they appear in the following list:

\begin{flalign*}
&(2s-1)!\operatorname{\Romannum{1}}_{2s}^{2p}(q)
=\sum_{r=1}^{p}t(2s,2r)\D^{2r-1}A_{2p-2r+1}(q^2)+\sum_{r=p+1}^{s}t(2s,2r)\D^{2p} A_{2r-2p-1}(q^2),& 
\end{flalign*}

\begin{flalign*}
(2s)!\operatorname{\Romannum{1}}_{2s+1}^{2p+1}(q)=&\sum_{r=0}^{p}2^{-2r}t(2s+1,2r+1)\D^{2r}C_{2p-2r+1}(q)\\
&+\sum_{r=p+1}^{s}2^{-2r}t(2s+1,2r+1)\D^{2p+1} F_{2r-2p-1}(q),& 
\end{flalign*}

\begin{flalign*}
&(2s-1)!(-1)^{s-1}\operatorname{\Romannum{2}}_{2s}^{2p}(q)=\sum_{r=1}^{p}t(2s,2r)\D^{2r-1} L_{2p-2r+1}(q^2)+\sum_{r=p+1}^{s}t(2s,2r)\D^{2p} B_{2r-2p-1}(q^2),& 
\end{flalign*}

\begin{flalign*}
(2s)!(-1)^s\operatorname{\Romannum{2}}_{2s+1}^{2p}(q)=&\sum_{r=0}^{p-1}2^{-2r}t(2s+1,2r+1)\D^{2r} N_{2p-2r}(q)\\
&+\sum_{r=p}^{s}2^{-2r}t(2s+1,2r+1)\D^{2p} G_{2r-2p}(q),&
\end{flalign*}

\begin{flalign*}
&(2s-1)! \operatorname{\Romannum{3}}_{2s}^{2p}(q)=\sum_{r=1}^{p}t(2s,2r)\D^{2r-1}F_{2p-2r+1}(q)+\sum_{r=p+1}^{s}t(2s,2r)\D^{2p} C_{2r-2p-1}(q),&
\end{flalign*}

\begin{flalign*}
(2s)!\operatorname{\Romannum{3}}_{2s+1}^{2p+1}(q)=&\sum_{r=0}^{p}t(2s+1,2r+1)\D^{2r}H_{2p-2r+1}(q)\\
&+\sum_{r=p+1}^{s}2^{2p+1-2r}t(2s+1,2r+1)\D^{2p+1} H_{2r-2p-1}(q),&
\end{flalign*}

\begin{flalign*}
&(2s-1)!(-1)^{s-1} \operatorname{\Romannum{4}}_{2s}^{2p}(q)=\sum_{r=1}^{p}t(2s,2r)\D^{2r-1}Q_{2p-2r+1}(q)+\sum_{r=p+1}^{s}t(2s,2r)\D^{2p} D_{2r-2p-1}(q),&
\end{flalign*}

\begin{flalign*}
(2s)!(-1)^s\operatorname{\Romannum{4}}_{2s+1}^{2p}(q)=&\sum_{r=0}^{p-1}t(2s+1,2r+1)\D^{2r} T_{2p-2r}(q)\\
&+\sum_{r=p}^{s}2^{2p-2r}t(2s+1,2m+1)\D^{2p} J_{2r-2p}(q),&
\end{flalign*}

\begin{flalign*}
&(2s-1)! \operatorname{\Romannum{5}}_{2s}^{2p}(q)=\sum_{r=1}^{p}t(2s,2r)\D^{2r-1}B_{2p-2r+1}(q^2)+\sum_{r=p+1}^{s}t(2s,2r)\D^{2p} L_{2r-2p-1}(q^2),&
\end{flalign*}

\begin{flalign*}
(2s)!\operatorname{\Romannum{5}}_{2s+1}^{2p+1}(q)=&\sum_{r=0}^{p}2^{-2r}t(2s+1,2r+1)\D^{2r}D_{2p-2r+1}(q)\\
&+\sum_{r=p+1}^{s}2^{-2r}t(2s+1,2r+1)\D^{2p+1} Q_{2r-2p-1}(q),&
\end{flalign*}

\begin{flalign*}
&(2s-1)!(-1)^{s-1}\operatorname{\Romannum{6}}_{2s}^{2p}(q)=\sum_{r=1}^{p}t(2s,2r)\D^{2r-1} M_{2p-2r+1}(q^2)+\sum_{r=p+1}^{s}t(2s,2r)\D^{2p} M_{2r-2p-1}(q^2),&
\end{flalign*}

\begin{flalign*}
(2s)!(-1)^s\operatorname{\Romannum{6}}_{2s+1}^{2p}(q)=&\sum_{r=0}^{p-1}2^{-2r}t(2s+1,2r+1)\D^{2r} P_{2p-2r}(q)\\
&+\sum_{r=p}^{s}2^{-2r}t(2s+1,2r+1)\D^{2p} R_{2r-2p}(q),&
\end{flalign*}

\begin{flalign*}
&(2s-1)!\operatorname{\Romannum{7}}_{2s}^{2p+1}(q)=\sum_{r=1}^{p}t(2s,2r)\D^{2r-1}G_{2p-2r+2}(q)+\sum_{r=p+1}^st(2s,2r)\D^{2p+1}N_{2r-2p-2}(q),&
\end{flalign*}

\begin{flalign*}
(2s)!\operatorname{\Romannum{7}}_{2s+1}^{2p}(q)=&\sum_{r=0}^{p-1}t(2s+1,2r+1)\D^{2r} J_{2p-2r}(q)\\
&+\sum_{r=p}^{s}2^{2p-2r}t(2s+1,2r+1)\D^{2p} T_{2r-2p}(q),&
\end{flalign*}

\begin{flalign*}
&(2s-1)!(-1)^{s-1}\operatorname{\Romannum{8}}_{2s}^{2p+1}(q)=&\sum_{r=1}^{p}t(2s,2r)\D^{2r-1} R_{2p-2r+2}(q)+\sum_{r=p+1}^{s}t(2s,2r)\D^{2p+1} P_{2r-2p-2}(q),&
\end{flalign*}

\begin{flalign*}
(2s)!(-1)^s\operatorname{\Romannum{8}}_{2s+1}^{2p+1}(q)=&\sum_{r=0}^{p}t(2s+1,2r+1)\D^{2r}U_{2p-2r+1}(q)\\
&+\sum_{r=p+1}^{s}2^{2p-2r+1}t(2s+1,2r+1)\D^{2p+1} U_{2r-2p-1}(q).&
\end{flalign*}
\end{theorem}

\section{CM evaluation of modular forms and quasimodular forms}\label{section_cmvalue}
\subsection{The general methods}
The $j$-invariant is a function defined on the upper half plane by
\[
j(z)=1728\frac{E_4^3(z)}{E_4^3(z)-E_6^2(z)},
\]
which has $q$-expansion
\[
j(z)=q^{-1}+744+196884q+21493760q^2+\cdots.
\]
It is well-known that if $z$ is a CM point, then $j(z)$ is an algebraic integer. These special values are called singular moduli and there has been extensive research on singular moduli, for example see \cite{MR0772491}.
For a general modular function for subgroup of finite index of $\operatorname{SL}_2(\mathbb{Z})$, its values at CM points are algebraic numbers if its Fourier coefficients are algebraic. The reason is that any such modular function and the function $j(z)$ are algebraically dependent over $\overline{\mathbb{Q}}$. 

Another important function is the Dedekind eta function which is defined by an infinite product:
\[
\eta(z)=q^{1/24}\prod_{n=1}^{\infty}\left(1-q^n\right),\quad q=e^{2\pi iz}.
\]
Ramanujan evaluated some of its values, for example
\[
\begin{aligned}
&\eta(i)=\frac{\Gamma(\frac{1}{4})}{2\pi^{3/4}},\quad
&\eta(2i)&=\frac{1}{2^{3/8}}\frac{\Gamma(\frac{1}{4})}{2\pi^{3/4}},\\
&\eta(4i)=\frac{1}{2^{13/16}(1+\sqrt{2})^{1/4}}\frac{\Gamma(\frac{1}{4})}{2\pi^{3/4}},\quad
&\eta(8i)&=\frac{1}{2^{41/32}}\frac{(-1+\sqrt[4]{2})^{1/2}}{(1+\sqrt{2})^{1/8}}\frac{\Gamma(\frac{1}{4})}{2\pi^{3/4}},
\end{aligned}
\]
see \cite[p.326]{MR1486573}. For a general CM point $z_0$, let $K$ be the quadratic field $\mathbb{Q}(z_0)$ and $D$ be its discriminant. The Chowla–Selberg formula asserts that
\[
\operatorname{Im}(z_0)|\eta(z_0)|^4=\frac{\alpha}{4\pi\sqrt{|D|}}\prod_{r=1}^{|D|-1}\Gamma(r/|D|)^{\left(\frac{D}{r}\right)\frac{w}{2h}},
\]
where $\alpha$ is an algebraic number, $h$ is the class number of $K$, and $w$ is the number of roots of unity, see \cite{MR0215797}. Inspired by the Chowla–Selberg formula, we define the Chowla–Selberg period by
\[
\Omega_K=\frac{1}{\sqrt{2\pi |D|}}\prod_{r=1}^{|D|-1}\Gamma(r/|D|)^{\left(\frac{D}{r}\right)\frac{w}{4h}}.
\]
Now let $f$ be a modular form of weight $k$ with algebraic Fourier coefficients. Then the function $f/\eta^{2k}$ is a modular function for some subgroup of finite index of $\operatorname{SL}_2(\mathbb{Z})$. Hence its value at $z_0$ is a algebraic number. This implies that for any modular form of weight $k$ with algebraic Fourier coefficients, its value at $z_0$ is an algebraic number multiplied by $\Omega_K^k$. However, finding this algebraic number precisely is not an easy task. Van der Poorten and Williams \cite{MR1692895} determined explicitly $|\eta(z)|$ for certain CM points $z$ . The importance of evaluating the Dedekind eta function is that many modular forms can be expressed as combinations of products of the eta function. For example,
\[\begin{aligned}
E_4(z)&=\frac{\eta^{16}(z)}{\eta^8(2z)}+2^8\frac{\eta^{16}(2z)}{\eta^8(z)},\\
E_6(z)&=\frac{\eta^{24}(z)}{\eta^{12}(2z)}-2^5\cdot 3\cdot 5\cdot\eta^{12}(2z)-2^9\cdot 3\cdot 11\cdot \frac{\eta^{12}(2z)\eta^8(4z)}{\eta^8(z)}+2^{13}\cdot \frac{\eta^{24}(4z)}{\eta^{12}(2z)},\\
H(z)&:=\sum_{n=0}^{\infty}\sigma_1(2n+1)q^{2n+1}=\frac{\eta^8(4z)}{\eta^4(2z)},
~G(z):=4E_1^{\chi_{4},1}(z)=\frac{\eta^{10}(2z)}{\eta^4(z)\eta^4(4z)}.
\end{aligned}
\]
The functions $H$ and $G$ will be used in the calculation of the series in the second class defined in Section \ref{section_derandser}. Using those eta expressions, one can determine the values of the above modular forms at CM points if one knows certain values of the eta function. Furthermore, since $E_4^3(z)/\eta^{24}(z)=j(z)$ and $E_6^2(z)/\eta^{24}(z)=j(z)-1728$, it follows that $E_4(z_0)/\eta^8(z_0)$ and $E_6(z_0)/\eta^{12}(z_0)$ are algebraic integers. Under certain conditions, these numbers lie in the ring class fields of orders in the filed $\mathbb{Q}(z_0)$, see \cite[p.195]{MR4502401} and the references given there. This fact, coupled with some numerical computations, allows for the calculation of the values of $E_4$ and $E_6$ at some CM points.  For the algebraic property of $E_2$ at CM points, see \cite[Lemma 2]{MR4372241}.

The modular polynomials can help us determining the values of modular forms at CM points. Given $f\in M_k(\Gamma_0(Nn),\chi)$, define the transformation equation
\[
\Psi^f_n(X,z)=\prod_{\gamma\in\Gamma_0(nN)\setminus \Gamma_0(N)}(X-\overline{\chi}(\gamma)f|_{k}\gamma)=\sum_{j=0}^{d}a_j(z)X^{d-j},
\]
where $d=[\Gamma_0(nN):\Gamma_0(N)]=n\prod_{p\mid n,p\nmid N}(1+p^{-1})$. It can be verified that this definition is well-defined, and furthermore, $a_j(z)\in M_{jk}(\Gamma_0(N),\chi^j)$. If we choose a set of generators for the space $M_{*}(\Gamma_1(N))$, then the coefficients of this polynomial can be expressed in terms of these generators. Thus it gives an algorithm to compute values of $f(z)$ if values of the generators are known. If $f(z)=g(nz)$ for some $g\in M_{k}(\Gamma_0(N),\chi)$ and $\chi$ is a Dirichlet character modulo $N$, the coefficients of $\Psi^f_n$ can be expressed in terms of the Hecke operators on $g(z)$, as illustrated by the following proposition.
\begin{proposition}\label{prpo_mod_poly}
With the above notation and assumptions, if $\gcd(n,N)=1$, the coefficients $a_j(z)$ are recursively given by
\[
a_j(z)=-\frac{1}{j}\sum_{r=1}^j\overline{\chi}^r(n)n^{1-kr}T_n^{*}(g^r)(z)a_{j-r}(z),
\]
where $T_n^{*}$ is related to the usual Hecke operator by
\[
T_n^{*}=\sum_{d^2\mid n}\mu(d)d^{k-2}T_{n/d^2}=\prod_{p\| n}T_p\prod_{\substack{p^{\alpha}\|n\\
\alpha\geq 2}}(T_{p^{\alpha}}-p^{k-2}T_{p^{\alpha-2}}).
\]
\end{proposition}
\begin{proof}
We briefly outline the proof of this proposition. Newton's identities allow expressing the coefficients of the polynomial $\Psi^f_n$ in terms of the sums of the powers of its roots.  If we denote the sum of the $r$-th power of the roots by
\[
p_r(z)=\sum_{\gamma\in\Gamma_0(nN)\setminus\Gamma_0(N)}(\overline{\chi}(\gamma)f|_k\gamma)^r,
\]
Newton's identities state that
\[
ja_j(z)=-\sum_{r=1}^ja_{j-r}(z)p_r(z).
\]
Hence to prove the proposition, it is sufficient to show that
\[
\overline{\chi}(n)n^{1-k}T_n^{*}(g)=\sum_{\gamma\in\Gamma_0(nN)\setminus\Gamma_0(N)}\overline{\chi}(\gamma)f|_{k}\gamma \text{ for }g\in M_k(\Gamma_0(N),\chi) \text{ and }f(z)=g(nz).
\]
Let
\[
\begin{aligned}
\Gamma_n(N)&=\left\{\begin{pmatrix}
a&b\\
c&d
\end{pmatrix}\in \operatorname{M}_2(\mathbb{Z}),~ad-bc=n,N\mid c,\gcd(a,N)=1\right\},\\
\Gamma^{*}_n(N)&=\left\{\begin{pmatrix}
a&b\\
c&d
\end{pmatrix}\in \Gamma_n(N),~\gcd(a,b,c,d)=1\right\}.
\end{aligned}
\]
It is not hard to verify that $\Gamma_0(N)$ acts from left on $\Gamma_n(N)$ and $\Gamma^{*}_n(N)$. Recall that the usual Hecke operator of level $N$ is defined by
\[
T_ng=n^{k/2-1}\sum_{\gamma\in\Gamma_0(N)\setminus\Gamma_n(N)}\overline{\chi}(\gamma)g |_k \gamma,
\]
and we define another Hecke operator by
\[
T_n^{*}g=n^{k/2-1}\sum_{\gamma\in\Gamma_0(N)\setminus\Gamma^{*}_n(N)}\overline{\chi}(\gamma)g |_k \gamma.
\]

By the double coset decompositions
\[
\Gamma^{*}_n(N)=\Gamma_0(N)\begin{pmatrix}
1&0\\
0&n
\end{pmatrix}\Gamma_0(N) \text{ and }\Gamma_n(N)=\bigsqcup_{\substack{ad=n,a\geq 1\\ a\mid d,\gcd(a,N)=1}}\Gamma_0(N)\begin{pmatrix}
a&0\\
0&d
\end{pmatrix}\Gamma_0(N),
\]
these two operators are related by
\[
T_n=\sum_{d^2\mid n}d^{k-2}T_{n/d^2}^{*},
\]
and the inversion of the above relation is 
\[
T_n^{*}=\sum_{d^2\mid n}\mu(d)d^{k-2}T_{n/d^2}.
\]
Since $f(z)=g(nz)=n^{-k/2}g|_k\small{\begin{pmatrix}n&0\\0&1\end{pmatrix}}$, we have
\[
\sum_{\gamma\in\Gamma_0(nN)\setminus\Gamma_0(N)}\overline{\chi}(\gamma)f|_{k}\gamma=n^{-k/2}\sum_{\gamma\in\Gamma_0(nN)\setminus\Gamma_0(N)}\overline{\chi}(\gamma)g|_{k}\small{\begin{pmatrix}n&0\\0&1\end{pmatrix}}\gamma.
\]
It suffices to show that the map
\[
 \Gamma_0(nN)\setminus\Gamma_0(N)\rightarrow\Gamma_0(N)\setminus\Gamma^{*}_n(N):\gamma\rightarrow \begin{pmatrix}n&0\\0&1\end{pmatrix}\gamma
\]
gives a bijection. Let $\gamma=\begin{pmatrix}a&b\\c&d\end{pmatrix}\in \Gamma_0(N)$, note that $\begin{pmatrix}na&nb\\c&d\end{pmatrix}\in \Gamma_n^{*}(N)$. Indeed, it has determinant $n$, we have $N\mid c$ and $\gcd(na,N)=1$ since $ad-bc=1$, and $\gcd(na,nb,c,d)=1$ by $\gcd(n,N)=1$. Now let $\alpha\in \Gamma_0(nN)$, by the identity
\[
\begin{pmatrix}n&0\\0&1\end{pmatrix}\alpha\gamma=\begin{pmatrix}n&0\\0&1\end{pmatrix}\alpha\begin{pmatrix}n^{-1}&0\\0&1\end{pmatrix}\begin{pmatrix}n&0\\0&1\end{pmatrix}\gamma
\]
and $\begin{pmatrix}n&0\\0&1\end{pmatrix}\alpha\begin{pmatrix}n^{-1}&0\\0&1\end{pmatrix}\in \Gamma_0(N)$, the map indeed sends left orbits to left orbits. And since  $|\Gamma_0(nN)\setminus\Gamma_0(N)|=|\Gamma_0(N)\setminus\Gamma^{*}_n(N)|=d$, it remains to show that this map is an injection. If $\begin{pmatrix}n&0\\0&1\end{pmatrix}\gamma_1=\alpha \begin{pmatrix}n&0\\0&1\end{pmatrix}\gamma_2$ for some $\alpha\in \Gamma_0(N)$, thus $\gamma_1=\begin{pmatrix}n^{-1}&0\\0&1\end{pmatrix}\alpha \begin{pmatrix}n&0\\0&1\end{pmatrix}\gamma_2$, and $\begin{pmatrix}n^{-1}&0\\0&1\end{pmatrix}\alpha \begin{pmatrix}n&0\\0&1\end{pmatrix}\in \Gamma_0(nN)$, it completes the proof.
\end{proof} 

We remark that from the above proof and the fact $f(z)\in M_k(\Gamma_0(N),\chi)$, when $\gcd(n,N)=1$, the roots of the polynomial $\Psi^{f|_{V_n}}_n(X,z)$ can be expressed by the set
\[
\left\{\overline{\chi}(n)\overline{\chi}(\gamma)n^{-k/2}(f|_k\gamma)(z),\gamma\in \Gamma_0(N)\setminus \Gamma_n^{*}(N) \right\}.
\]
Modular polynomials can assist us in computing CM values of modular forms.
We illustrate the idea using the space of modular forms for $\operatorname{SL}_2(\mathbb{Z})$ as an example. Since the graded ring of modular forms for $\operatorname{SL}_2(\mathbb{Z})$ is isomorphic to $\mathbb{C}[E_4,E_6]$. To compute CM values of modular forms for $\operatorname{SL}_2(\mathbb{Z})$, it is sufficient to calculate the CM values of $E_4$ and $E_6$. 
By the above remark, $\{(cz_0+d)^{-k}E_k(\gamma z_0),\gamma=\begin{pmatrix}a&b\\c&d\end{pmatrix}\in \operatorname{SL}_2(\mathbb{Z})\setminus \Gamma_n^{*}(1)\}$ are roots of $\Psi^{E_k|_{V_n}}_n(X,z_0)$, $k=4$ and 6.  Since the coefficients of $\Psi^{E_k|_{V_n}}_n(X,z_0)$ are polynomials of $E_4$ and $E_6$, if we know the values of $E_4(z_0)$ and $E_6(z_0)$, solving the equation $\Psi^{E_k|_{V_n}}_n(X,z_0)$ and combining some numerical calculations, we can determine these values. In particular, a system of representatives of the left action of $\operatorname{SL}_2(\mathbb{Z})$ on $\Gamma_n^{*}(1)$ is 
\[
\left\{\begin{pmatrix}
a&b\\
0&d
\end{pmatrix}\in \operatorname{M}_2(\mathbb{Z}),ad=n,0\leq b\leq d-1,~\gcd(a,b,d)=1 \right\}.
\]
Thus if the values of $E_4(z_0)$ and $E_6(z_0)$ are known, with the help of the modular polynomials, we can determine the values of $E_k((az_0+b)/d),$ for $ad=n,~0\leq b\leq d-1$, $k=4$ and 6.

Assume that we have calculated the values of $E_4(z_0)$ and $E_6(z_0)$. This time we consider the modular polynomial of $nE_2(nz)-E_2(z)$. We claim that the roots of the polynomial $\Psi_n^{nE_2|_{V_n}-E_2}(X,z)$ are
\[
\left\{n(cz+d)^{-2}E_2(\eta z)-E_2(z)-\frac{6}{\pi i}\frac{c}{cz+d},\eta=\begin{pmatrix}
a&b\\
c&d
\end{pmatrix}\in \operatorname{SL}_2(\mathbb{Z})\setminus \Gamma_n^{*}(1)\right\}.
\]
This is not a trivial fact, since by the definition of the modular polynomials, the roots of $\Psi_n^{nE_2|_{V_n}-E_2}(X,z)$ are $\left\{(cz+d)^{-2}(nE_2(n\gamma z)-E_2(\gamma z)),~\gamma=\begin{pmatrix}
a&b\\
c&d
\end{pmatrix}\in \Gamma_0(n)\setminus \operatorname{SL}_2(\mathbb{Z})\right\}$. By the proof of Proposition \ref{prpo_mod_poly}, there exists a one-to-one correspondence between the two cosets $\Gamma_0(n)\setminus\operatorname{SL}_2(\mathbb{Z})$ and $\operatorname{SL}_2(\mathbb{Z})\setminus\Gamma_n^{*}(1)$, which is given by $\gamma\rightarrow v_n\gamma$, where $v_n$ denotes the matrix $\small\begin{pmatrix}
n&0\\
0&1
\end{pmatrix}$. Therefore, any element $\eta$ in a representative class of the right coset $\operatorname{SL}_2(\mathbb{Z})\setminus \Gamma_n^{*}(1)$ can be represented as $\eta=\tau v_n\gamma$, where $\tau,\gamma\in\operatorname{SL}_2(\mathbb{Z})$.  It is sufficient to verify the identity
\[
n(c_{\eta}z+d_{\eta})^{-2}E_2(\eta z)-E_2(z)-\frac{6}{\pi i}\frac{c_{\eta}}{c_{\eta}z+d_{\eta}}
=(c_{\gamma}z+d_{\gamma})^{-2}(nE_2(v_n\gamma z)-E_2(\gamma z)), 
\]
for $\eta=\tau v_n\gamma$. For notation simplicity, let $j(\gamma,z)=c_{\gamma}z+d_{\gamma}$ for any $\gamma\in \operatorname{M}_{2}(\mathbb{Z})$. Note that $j(\alpha\beta,z)=j(\alpha,\beta z)j(\beta,z)$. Using the transformation property of $E_2$, we get
\[
\begin{aligned}
j(\tau v_n\gamma,z)^{-2}E_2(\tau v_n\gamma z)&=j(\tau v_n\gamma,z)^{-2}j(\tau,v_n\gamma z)^{2}E_2(v_n\gamma z)+\frac{6}{\pi i}c_{\tau}j(\eta,z)^{-2}j(\tau,v_n\gamma z)\\
&=j(\gamma,z)^{-2}E_2(v_n\gamma z)+\frac{6}{\pi i}c_{\tau}j(\tau v_n\gamma,z)^{-1}j(\gamma,z)^{-1}.
\end{aligned}
\]
So it suffices to show that
\[
c_{\eta}j(\eta,z)^{-1}=c_{\gamma}j(\gamma,z)^{-1}+nc_{\tau}j(\eta,z)^{-1}j(\gamma,z)^{-1},
\]
and this follows easily by the assumption $\eta=\tau v_n\gamma$. Thus if $z_0$ is a CM point such that $\eta z_0=z_0$ for some $\eta\in \Gamma_n^{*}(1)$, the above discussion implies that $(nj(\eta,z_0)^{-2}-1)E_2(z_0)-\frac{6}{\pi i}c_{\eta} j(\eta,z_0)^{-1}$ is a root of $\Psi_n^{nE_2|_{V_n}-E_2}(X,z_0)$. This allows us to find the value of $E_2(z_0)$ if we know the values of $E_4(z_0)$ and $E_6(z_0)$. Moreover, knowing the values of $E_2(z_0)$, $E_4(z_0)$ and $E_6(z_0)$, we can calculate the values of $E_2((az_0+b)/d)$, for $ad=n$, $0\leq b\leq d-1$.

For computing the values of modular forms at CM points, we can also utilize the following polynomial. For $f\in M_k(\Gamma_0(N),\chi)$, define 
\[
\Phi^f_n(X,z)=\prod_{\gamma\in\Gamma_0(N)\setminus \Gamma_n^{*}(N)}\left(X-\overline{\chi}(\gamma)f|_k\gamma\right)=\sum_{j=0}^db_j(z)X^{d-j},
\]
the coefficients are given by
\[
b_j(z)=-\frac{1}{j}\sum_{r=1}^jn^{1-kr/2}b_{j-r}(z)T_n^{*}(f^r)(z).
\]
By Proposition \ref{prpo_mod_poly}, if $\gcd(n,N)=1$, the polynomials $\Psi$ and $\Phi$ are related by
\[
\Phi^f_n(X,z)=\chi^d(n)n^{kd/2}\Psi^{f|_{V_n}}_n(\overline{\chi}(n)n^{-k/2}X,z).
\]
This polynomial can also help  us in computing the values of modular forms at CM points, especially when $\gcd(n,N)>1$.

For our purpose, we also need to know CM values of derivatives of modular forms. We define the modified differentiation operator by
\[
\partial:=\partial_k=\D-\frac{k}{4\pi y}.
\]
The operator is essentially the Maass raising operator, scaled by a constant. The function $\partial_k f$ is no longer holomorphic but it transforms like a modular form. More generally, for integer $h\geq 1$, define $\partial_k^h$ by
\[
\partial^h:=\partial_k^h=\partial_{k+2h-2}\circ\partial_{k+2h-4}\circ\cdots\partial_{k+2} \circ\partial_{k},
\]
those operators satisfy that for modular form $f\in M_k(\Gamma_1(N))$,
\[
(\partial^h f|_{k+2h}\gamma)(z)=\partial^h f(z),~\gamma\in\Gamma_1(N).
\]
The modified differential operators and the usual differential operators are related by the following identity, which can be proven by induction,
\begin{equation}\label{equ_Dandpar}
\D^n f=\sum_{r=0}^n\binom nr\frac{(k+r)_{n-r}}{(4\pi y)^{n-r}}\partial^r f,
\end{equation}
where $(a)_m=a(a+1)\cdots(a+m-1)$. Shimura \cite{MR0491519} found that the non-holomorphic derivative $\partial$ preserves rationality property at CM points of modular forms with algebraic Fourier coefficients. He proved that for modular form $f$ of weight $k$, $\partial^r f(z_0)$ is an algebraic multiple of $\Omega_K^{k+2r}$. By (\ref{equ_Dandpar}), we have the following algebraic statement.
\begin{proposition}\label{prop_cmvalues}
Let $f$ be a modular form of weight $k$ with algebraic Fourier coefficients, and let $z_0$ be a CM point and $K=\mathbb{Q}(z_0)$, then
\[
\D^nf(z_0)\in \sum_{0\leq r\leq n}\overline{\mathbb{Q}}\cdot\Omega_K^{k+2r}\pi^{r-n}.
\]
In particular, since $\D \Delta=E_2\Delta$, the above result implies that 
\[
E_2(z_0)-\frac{3}{\pi y}\in \overline{\mathbb{Q}}\cdot\Omega_K^2,~y=\operatorname{Im}(z_0).
\]
\end{proposition}

This proposition does not provide us with the method for explicit computation. In Section \ref{section_higer}, we present an algorithm that enables us to calculate higher-order derivatives of modular forms.
\subsection{Computation: CM values of modular forms}\label{subsec_CMvalues}
In this subsection, we employ the aforementioned methods, particularly the modular polynomial method, to compute the series defined in Section \ref{section_derandser} at the points $z_0=i$, $2i$ and $i/2$. Let $\Omega:=\Omega_{\mathbb{Q}(i)}=\Gamma(\frac{1}{4})^2/4\pi^{3/2}$.

By Proposition \ref{prop_expressqseries} and
\[
A_{2s-1}(q)=\frac{B_{2s}}{4s}(1-E_{2s}(z)), q=e^{2\pi iz},
\]
to determine CM values of series in the first class, we only need to calculate the values of the function $E_{2s}$. And since the graded ring of modular forms of level 1 is isomorphic to $\mathbb{C}[E_4,E_6]$, every $E_{2s}, s\geq 2$ can be expressed by $E_4$ and $E_6$. In particular, let $F_n(z)=\frac{-B_n}{2n(n-2)!}E_n(z)$, Ramanujan \cite{ramanujan1916certain} gave the recurrence relation
\begin{equation}\label{equ_Ramrecurr}
(n-2)(n+5)F_{n+4}=12(F_4F_n+F_6F_{n-2}+\cdots+F_nF_4).
\end{equation}

So it suffices to compute the CM values of $E_4$ and $E_6$. By the transformation property  
\[
E_2\left(-\frac{1}{z}\right)-\frac{3}{\pi\operatorname{Im}(-1/z)}=z^2\left(E_2(z)-\frac{3}{\pi \operatorname{Im}(z)}\right)\text{ and }E_6\left(-\frac{1}{z}\right)=z^6E_6(z)
\]
and since $i$ is fixed under $z\rightarrow -1/z$, we see that $E_2(i)=3/\pi$ and $E_6(i)=0$. We also need one evaluation of the eta function: $\eta(i)^2=\Omega$. Combining $1728\eta^{24}(z)=E_4^3(z)-E_6^2(z)$, we deduce that $E_4(i)=12\Omega^4$.

Using Proposition \ref{prpo_mod_poly}, we can calculate some modular polynomials as follows 
\[
\Psi^{E_4|_{V_2}}_2(X,z)=X^3-\frac{9}{8}E_4(z)X^2+\frac{33}{256}E^2_4(z)X+\frac{121}{1024}E_4^3(z)-\frac{125}{1024}E_6^2(z),
\]
\[
\begin{aligned}
\Psi^{E_6|_{V_2}}_2(X,z)&=X^3-\frac{33}{32}E_6(z)X^2+\left(\frac{1452}{4096}E_6^2(z)-\frac{1323}{4096}E_4^3(z)\right)X\\
&-\frac{1331}{32768}E_6^3(z)+\frac{1323}{32768}E_6(z)E_4^3(z),
\end{aligned}
\]
\[
\Psi^{2E_2|_{V_2}-E_2}_2(X,z)=X^3-\frac{3}{4}E_4(z)X-\frac14 E_6(z).
\]

Substituting $E_4(i)/\Omega^4=12$ and $E_6(i)=0$ into the modular polynomial $\Psi^{E_4|{V_2}}_2$, we get
\[
x^3-\frac{27}{2}x^2+\frac{297}{16}x+\frac{3267}{16}=\frac{1}{16}(4x - 33)^2(x + 3)
=0,
\]
and three roots of this polynomial are $2^{-4}E_4(\frac{i}{2})\Omega^{-4}$,$2^{-4}E_4(\frac{i+1}{2})\Omega^{-4}$ and $E_4(2i)\Omega^{-4}$ by remark after Proposition \ref{prpo_mod_poly}. Some numerical calculations show that
\[
E_4(2i)=\frac{33}{4}\Omega^4,\quad E_4\left(\frac{i}{2}\right)=132\Omega^4,\quad E_4\left(\frac{1+i}{2}\right)=-48\Omega^4.
\]
For our computation of $U_s(i/2)$, we also need the values of $E_s$ at $(1\pm i)/4$. This can be easily derived from the transformation property of the Eisenstein series, for example
\[
E_4\left(\frac{1\pm i}{4}\right)=\left(-\frac{4}{i\pm 1}\right)^4E_4(2i)=-528.
\]

Similar to the above procedure, we can compute the values of $E_2$ and $E_6$. We put some values in the table below. Here $E_2^{*}(z):=E_2(z)-3/(\pi \operatorname{Im}z)$.

\begin{center}
\begin{tabular}{l|l|l|l}
  \toprule
  $z_0$ & $E_2^{*}(z_0)\Omega^{-2}$  & $E_4(z_0)\Omega^{-4}$  & $E_6(z_0)\Omega^{-6}$ \\
  \midrule
  $i$    & 0  & 12 & 0 \\
  $2i$  &  $\frac{3}{2}$  & $\frac{33}{4}$  & $\frac{189}{8}$ \\
  $4i$   & $\frac{9}{8}+\frac{3}{4}\sqrt{2}$ & $\frac{45}{16}\sqrt{2}+\frac{273}{64}$ & $\frac{2079}{256}\sqrt{2}+\frac{6237}{512}$ \\
  $(1\pm i)/4$ &$\mp 12i$ &-528& $\pm 12096 i$\\
  $i/2$ &-6 &132&-1512\\
  $i/4$ &$-18-12\sqrt{2}$ &$720\sqrt{2}+1092$&$-33264\sqrt{2}-49896$\\
  \bottomrule
  
\end{tabular}
\end{center}

For $q$-series in the second class, let
\[
\psi_k=
\left\{
\begin{aligned}
\chi_{-4} \qquad\text{ if }&k\text{ is odd,}\\
1  \qquad\text{ if }&k\text{ is even.}
\end{aligned}\right.
\]
As a graded algebra, $\bigoplus_{k\geq 1}M_{k}(\Gamma_0(4),\psi_k)\cong \mathbb{C}[G,H]$, where $G=2E_{1}^{1,\chi_{-4}}(z)\in M_1(\Gamma_0(4),\chi_{-4})$ and $H=\sum_{n=0}^{\infty}\sigma_1(2n+1)q^{2n+1}=\sum_{n=1}^{\infty}\frac{(2n-1)q^{2n-1}}{1-q^{4n-2}}\in M_2(\Gamma_0(4))$, cf. \cite[p.14]{ono2004web}. This means that every modular form in $M_k(\Gamma_0(4),\psi_k)$ can be written as polynomials in terms of $G$ and $H$. In particular, the two Eisenstein series $E_{2s+1}^{1,\chi_{-4}}(z)$ and $E_{2s+1}^{\chi_{-4},1}(z)$  are linear combinations of $\{G^{2s+1},G^{2s-1}H,G^{2s-3}H^2,\cdots,GH^{s}\}$. There are also recursive formulas for $E_{2s+1}^{1,\chi_{-4}}(z)$ and $E_{2s+1}^{\chi_{-4},1}(z)$ similar to the Ramanujan's recursive formula (\ref{equ_Ramrecurr}), see \cite{MR2670979}. By Proposition \ref{prop_expressqseries}, to evaluate $q$-series in the second class, it is sufficient to evaluate $E_{2s+1}^{1,\chi_{-4}}(z)$ and $E_{2s+1}^{\chi_{-4},1}(z)$, which in turn reduces to evaluating $G(z)$ and $H(z)$. Using the identities
\[
3G^2(z)=4E_2(4z)-E_2(z),~16H(z)=2E_2(2z)-E_2(z)-G^2(z),
\]
and the modular polynomials
\[
\begin{aligned}
\Psi^{G|_{V_2}}_2(X,z)&=X^2-G(z)X+4H(z),\\
\Psi^{H|_{V_2}}_2(X,z)&=X^2+\left(-\frac1{16}G^2(z)+\frac12 H(z)\right)X+\frac 1{16}H^2(z),
\end{aligned}
\]
we calculate some of the values in the table below.
\begin{center}
\begin{tabular}{l|l|l}
  \toprule
  $z_0$ & $G(z_0)\Omega^{-1}$  & $H(z_0)\Omega^{-2}$ \\
  \midrule
  $i$    & $1+\frac{\sqrt{2}}2$  & $(3-2\sqrt{2})2^{-5}$  \\
  $i/2$ &2 &$\frac{1}{8}$ \\
  $i/4$ & $\sqrt{6+4\sqrt{2}}$ &$\frac{\sqrt{2}}{2}$\\
  $2i$ & $2^{-1/4}+2^{-1/2}+2^{-1}$ & $(2^{1/4}-1)^42^{-7}$\\
  \bottomrule
\end{tabular}
\end{center}
\subsection{Computation: higher-order derivatives of modular forms}\label{section_higer}
In this subsection, we provide an algorithm for recursively computing higher-order derivatives of modular forms and quasimodular forms. For the first class of series, computing the values of their higher-order derivatives at CM points only requires knowing the values of $E_2,E_4,E_6$ at CM points. Similarly, for the second class of series, determining the values of their higher-order derivatives at CM points only need knowing the values of $E_2,G,H$ at CM points.

To begin with, we note that the graded ring of quasimodular forms for $\operatorname{SL}_2(\mathbb{Z})$ is isomorphic to $\mathbb{C}[E_2,E_4,E_6]$. The following proposition gives an algorithm to calculate higher-order derivatives of quasimodular forms of level one.
\begin{proposition}\label{prop_diff_one}
Let $f$ be a quasimodular form of weight $k$ for $\operatorname{SL}_2(\mathbb{Z})$ with Fourier coefficients in a number field $\mathbb{K}$. By the structure of the graded ring of quasimodular forms, $f$ can be written as $E_2^{k/2}f_0(E_4E_2^{-2},E_6E_2^{-3})$ with $f_0(x,y)\in \mathbb{K}[x,y]$. Then $n$-th derivative of $f$ can be written as
$\D^n f=E_2^{k/2+n}f_n(E_4E_2^{-2},E_6E_2^{-3})$, and $f_n(x,y)\in \mathbb{K}[x,y]$ are given recursively by
\[
f_{n+1}(x,y)=\partial_{k+2n}\circ \partial_{k+2n-2}\cdots \circ\partial_{k} \left(f_0(x,y)\right),
\]
where $\partial_{k}$ is the  differential operator defined by
\[
\partial_{k}=\frac{k(1-x)}{24}+
\frac{x^2+x-2y}{6}\frac{\partial}{\partial x}+\frac{xy-2x^2+y}{4}\frac{\partial}{\partial y}.
\]
\end{proposition}
\begin{proof}
By direct computation, we get the differential equations
\[
\D E_2=\frac{E_2^2-E_4}{12},\D E_4=\frac{E_2E_4-E_6}{3}, \D E_6= \frac{E_2E_6-E_4^2}{2},
\]
so the differential operator $\D:\mathbb{K}[E_2,E_4,E_6]\rightarrow \mathbb{K}[E_2,E_4,E_6]$ can be written by
\[
\D=\frac{E_2^2-E_4}{12}\frac{\partial}{\partial E_2}+\frac{E_2E_4-E_6}{3}\frac{\partial}{\partial E_4}+\frac{E_2E_6-E_4^2}{2}\frac{\partial}{\partial E_6}.
\]
Since $\D ^nf$ is of weight $k+2n$, by homogeneity, there is a polynomial $f_n(x,y)\in \mathbb{K}[x,y]$ such that $\D ^nf=E_2^{k/2+n}f_n(E_4/E_2^2,E_6/E_2^3)$. Applying the operator $\D$ to $\D ^nf$,  We get 
\[
\begin{aligned}
\D^{n+1}f&=E_2^{k/2+n+1}f_{n+1}(E_4E_2^{-2},E_6E_2^{-3})=\D (E_2^{k/2+n}f_{n}(E_4E_2^{-2},E_6E_2^{-3}))\\
&=\left(\frac{k+2n}{2}\right)E_2^{k/2+n-1}\frac{E_2^2-E_4}{12}f_n(E_4E_2^{-2},E_6E_2^{-3})+E_2^{k/2+n}\D f_n(E_4E_2^{-2},E_6E_2^{-3}).
\end{aligned}
\]
Let $x=E_4E_2^{-2}$, $y=E_6E_2^{-3}$, we have
\[
\frac{\partial}{\partial E_2}=-2E_4E_2^{-3}\frac{\partial}{\partial x}-3E_6E_2^{-4}\frac{\partial}{\partial y},~
\frac{\partial}{\partial E_4}=E_2^{-2}\frac{\partial}{\partial x},~\frac{\partial}{\partial E_6}=E_2^{-3}\frac{\partial}{\partial y},
\]
hence the differential operator can be expressed by
\[
\D =\frac{E_4^2E_2^{-3}+E_4E_2^{-1}-2E_6E_2^{-2}}{6}\frac{\partial}{\partial x}+\frac{E_4E_6E_2^{-4}-2E_4^2E_2^{-3}+E_6E_2^{-2}}{4}\frac{\partial}{\partial y}.
\]
Using this expression, we have
\[
f_{n+1}(x,y)=\frac{k+2n}{24}(1-x)f_n(x,y)+\
\left(\frac{x^2+x-2y}{6}\frac{\partial}{\partial x}+\frac{xy-2x^2+y}{4}\frac{\partial}{\partial y}\right) f_n(x,y).
\]
If we define the differential operator by
\[
\partial_{k+2n}=\frac{(k+2n)(1-x)}{24}+
\frac{x^2+x-2y}{6}\frac{\partial}{\partial x}+\frac{xy-2x^2+y}{4}\frac{\partial}{\partial y},
\]
the above calculation implies that
\[
f_{n+1}(x,y)=\partial_{k+2n}\left(f_n(x,y)\right),
\]
which completes the proof.
\end{proof}
Since for a CM point $z_0$, $E_2(z_0)$ never equals to zero (Proposition \ref{prop_cmvalues}). The above proposition implies that $\D^n f(z_0)=E_2^{k/2+n}(z_0)f_n(E_4(z_0)E_2^{-2}(z_0),E_6(z_0)E_2^{-3}(z_0))$ and the polynomial $f_n$ can be explicitly calculated .

For the second class of series, we note that the graded ring of quasimodular forms for $\Gamma_0(4)$ with the multiplier system $\psi_k$ is isomorphic to $\mathbb{C}[E_2,G,H]$. Similarly as before, we have the following algorithm.
\begin{proposition}\label{prop_diff_two}
Let $f$ be a quasimodular form of weight $k$ for $\Gamma_0(4)$ with the multiplier system $\psi_k$ with Fourier coefficients in a number field $\mathbb{K}$, then $f=E_2^{k/2}f_0(GE_2^{-1/2},HE_2^{-1})$ with $f_0\in \mathbb{K}[x,y]$. We have
$\D^nf=E_2^{k/2+n}f_n(GE_2^{-1/2},HE_2^{-1})$, and $f_n(x,y)\in\mathbb{K}[x,y]$ are given  recursively by
\[
f_{n+1}(x,y)=\partial_{k+2n}\circ \partial_{k+2n-2}\cdots \circ\partial_{k} \left(f_0(x,y)\right),
\]
where $\partial_{k}$ is the  differential operator defined by
\[
\begin{aligned}
\partial_{k}=k\left(-\frac{x^4}{24}-\frac{28}{3}x^2y-\frac{32}{3}y^2+\frac{1}{24}\right)&+\left(\frac{x^5}{24}+\frac{28}{3}x^3y-\frac{x^3}{12}+\frac{32}{3}xy^2+\frac{20}{3}xy+\frac{x}{24}\right)\frac{\partial}{\partial x}\\
&+\left(\frac{x^4y}{12}+\frac{56}{3}x^2y^2+\frac{5}{6}x^2y+\frac{64}{3}y^3-\frac{8}{3}y^2+\frac{y}{12}\right)\frac{\partial}{\partial y}.
\end{aligned}
\]
\end{proposition}
\begin{proof}
Note that
\[
\begin{aligned}
&\D E_2=-\frac{1}{12}G^4-\frac{56}{3}G^2H-\frac{64}{3}H^2+\frac{1}{12}E_2^2,\\
&\D G=-\frac{1}{12}G^3+\frac{20}{3}GH+\frac{1}{12}GE_2,\\
&\D H=\frac{5}{6}G^2H-\frac{8}{3}H^2+\frac{1}{6}HE_2.
\end{aligned}
\]
The proof is similar as before. The details are left to the readers.
\end{proof}

\subsection{Explicit examples}
Using the computational data from the previous two sections, we can obtain the values of all the series given in Section \ref{section_intro} at $z=i$. In this subsection, we give detailed computations for a series using the machinery described above. 
\begin{example}
This example involves the series in the first class. We evaluate the reciprocal sum of hyperbolic functions $\sum_{n=1}^{\infty}n^4\operatorname{cosech}^8(n\pi)$, i.e. $\text{I}^4_8(e^{-\pi})$. By Theorem \ref{theorem_main}, 
\[
\text{I}^4_8(e^{-\pi})=\frac{1}{7!}\left(\D^4A_3-14\D^4 A_1+49\D^3 A_1-36\D A_3\right)(e^{-2\pi}).
\]
By Proposition \ref{prop_diff_one}, with the initial value of $f_0=1$, using the  recursive formula we obtain 
\[
f_3(x,y)=\partial_6\circ\partial_4\circ\partial_2(f_0)=-\frac1{96}x^2 - \frac1{48}x + \frac1{36}y + \frac{1}{288}
\]
and 
\[
f_4(x,y)=\partial_8(f_3)=-\frac{5}{288}x^2 + \frac1{216}xy - \frac5{432}x +\frac5{216}y + \frac{1}{864},
\]
this means
\[
\D^3 E_2=-\frac1{96}E_4^2  + \frac1{36}E_6E_2- \frac1{48}E_4E_2^2 + \frac{1}{288}E_2^4
\]
and
\[
\D^4 E_2= \frac1{216}E_4E_6-\frac{5}{288}E_4^2E_2 +\frac5{216}E_6E_2^2  - \frac5{432}E_4E_2^3 + \frac{1}{864}E_2^5.
\]
From the above expressions and values given in Section \ref{subsec_CMvalues}, we can compute 
\[
\D^3 E_2(i)=-\frac{3}{2}\Omega^8-\frac{9}{4}\frac{\Omega^4}{\pi^2}+\frac{9}{32}\frac{1}{\pi^4}\text{ and }
\D^4 E_2(i)=-\frac{15}{2}\frac{\Omega^8}{\pi}-\frac{15}{4}\frac{\Omega^4}{\pi^3}+\frac{9}{32}\frac1{\pi^5}.
\]
Similarly, we get
\[
\D E_4(i)=12\frac{\Omega^4}{\pi}
\text{ and }
\D^4 E_4(i)=30\Omega^{12}+315\frac{\Omega^8}{\pi^2}+\frac{315}{8}\frac{\Omega^4}{\pi^4}.
\]
Combining the above calculations, we have
\[
\begin{aligned}
&\sum_{n=1}^{\infty}n^4\operatorname{cosech}^8(n\pi)=\frac{2^8}{7!}\text{I}_8^4(e^{-\pi})=\frac{2^8}{7!}\left(\D^4A_3-14\D^4 A_1+49\D^3 A_1-36\D A_3\right)(e^{-2\pi})\\
&=\frac{2}{315}\Omega^{12}+\frac1{15}\frac{\Omega^8}{\pi^2}-\frac29\frac{\Omega^8}{\pi}+\frac{7}{45}\Omega^8+\frac1{120}\frac{\Omega^4}{\pi^4}-\frac19\frac{\Omega^4}{\pi^3}+
\frac{7}{30}\frac{\Omega^4}{\pi^2}-\frac{16}{175}\frac{\Omega^4}{\pi}+\frac{1}{120}\frac{1}{\pi^5}-\frac{7}{240}\frac{1}{\pi^4}.
\end{aligned}
\]
\end{example}

\begin{example}
Using the tools developed in this paper, one can verify the conjectures presented in \cite{xu2022functional}. Here, we provide an example as an illustration. This example addresses the third and fifth equations of Conjecture 1 in \cite{xu2022functional}, as well as the first and second equations of Conjecture 2 in \cite{xu2022functional}. Using our notation, their conjecture can be written as: Let $\Gamma=\Gamma(1/4)$ and $q=e^{-\pi}$, for any $p\in\mathbb{N}^{+}$, there exist $c_p,d_p,g_p,h_p\in\mathbb{Q}$ such that
\[
\operatorname{\Romannum{2}}_2^{4p-2}(q)=c_p\frac{\Gamma^{8p-4}}{\pi^{6p-2}}+d_p\frac{\Gamma^{8p}}{\pi^{6p}}\text{ and }\operatorname{\Romannum{3}}_2^{4p-2}(q)=g_p\frac{\Gamma^{8p-4}}{\pi^{6p-2}}+h_p\frac{\Gamma^{8p}}{\pi^{6p}},
\]
with 
\[
c_p=2^{2-4p}g_p \text{ and }d_p=-2^{2-4p}h_p.
\]

Using Theorem \ref{theorem_main} and Proposition \ref{prop_expressqseries}, we can write the series $\operatorname{\Romannum{2}}_2^{4p-2}(q)$ and $\operatorname{\Romannum{3}}_2^{4p-2}(q)$ as the special values of derivatives of Eisenstein series:
\begin{equation}\label{equation_expE}
\begin{aligned}
\operatorname{\Romannum{2}}_2^{4p-2}(q)&=\frac{B_{4p-2}}{8p-4}\left(4(\D E_{4p-2})(2i)-(\D E_{4p-2})(i)\right),\\
\operatorname{\Romannum{3}}_2^{4p-2}(q)&=\frac{B_{4p-2}}{8p-4}\left(2^{4p-2}(\D E_{4p-2})(i)-(\D E_{4p-2})(i/2)\right).
\end{aligned}
\end{equation}
The above series can be written as rational linear combinations of $E_2, E_4$ and $E_6$ by Proposition \ref{prop_diff_one}, and since the values of $E_2, E_4$ and $E_6$ at the relevant points ($z=i,2i,i/2$) are calculated in Section \ref{subsec_CMvalues}, the first claim follows directly. To establish the second claim, a more detailed analysis is required. %From
%\[
%E_{4p-2}\left(-\frac{1}{z}\right)=z^{4p-2} E_{4p-2}(z)\text{ and }\D E_{4p-2}\left(-\frac{1}%{z}\right)=z^{4p}\D E_{4p-2}(z)+\frac{2p-1}{\pi i}z^{4p-1}E_{4p-2}(z),
%\]
%we deduce that $E_{4p-2}(i)=0$ and $(\D E_{4p-2})(i/2)=2^{4p}(\D E_{4p-2})(2i)+(2p-1)\pi^{-1}2^{4p-2}E_{4p-2}(2i)$. 
Recall that the Serre derivative $\vartheta$ sends modular form $f$ of weight $k$ to modular form of weight $k+2$ which is defined by 
\[
\vartheta f=\D f-\frac{k}{12}E_2f.
\]
Note that for modular form of weight $k\equiv 2\mod 4$, its value at $i$ is zero by its transformation property. By replacing $\D E_{4p-2}$ in (\ref{equation_expE}) with $\vartheta E_{4p-2}$ and $E_2$, and combining with the  calculations of $E_2$ at $z=i,i/2$,  we obtain
\begin{align*}
\frac{8p-4}{B_{4p-2}}\operatorname{\Romannum{2}}_2^{4p-2}(q)&=4(\vartheta E_{4p-2})(2i)-(\vartheta E_{4p-2})(i)+\frac{(2p-1)\Gamma^4}{16\pi^3}E_{4p-2}(2i)\\
&+\frac{2p-1}{\pi}E_{4p-2}(2i),\\
\frac{8p-4}{B_{4p-2}}\operatorname{\Romannum{3}}_2^{4p-2}(q)&=2^{4p-2}(\vartheta E_{4p-2})(i)-(\vartheta E_{4p-2})(i/2)+\frac{(2p-1)\Gamma^4}{16\pi^3}(E_{4p-2})(i/2)\\
&-\frac{2p-1}{\pi}E_{4p-2}(i/2).
\end{align*}
The second claim then follows by 
\[2^{4p}(\vartheta E_{4p-2})(2i)=(\vartheta E_{4p-2})(i/2) \text{ and } -2^{4p-2}E_{4p-2}(2i)=E_{4p-2}(i/2).
\]
\end{example}

We can also use the above methods to compute the values of the series at CM points in other quadratic fields. We do not perform specific calculation here. Instead, we list some values in the following table. In the table below, $D$ is the discriminant of the quadratic field $K:=\mathbb{Q}(z_0)$.
\begin{center}
\begin{tabular}{l|l|l|l}

  \toprule
  $z_0$ & $\frac{|D|^{1/2}E_2^{*}(z_0)}{\Omega_K^{2}}$  & $\frac{E_4(z_0)}{\Omega_K^{4}}$  & $\frac{E_6(z_0)}{|D|^{1/2}\Omega_K^{6}}$ \\
  \midrule
  $\frac{1+\sqrt{3}i}{2}$    & 0  & 0 & $24$  \\
  $\frac{1+\sqrt{7}i}{2}$  &  $3$  & 15  & $27$ \\
  $\sqrt{2}i$   & $2$ & 20 & $28$ \\
  $\frac{1+\sqrt{11}i}{2}$ & $8$& 32 & $56$  \\
  $\frac{1+\sqrt{19}i}{2}$ & $24$ & 96 & $216$  \\
  $\frac{1+\sqrt{43}i}{2}$ & 144 & 960 & 4536 \\
  $\frac{1+\sqrt{67}i}{2}$& 456 & 5280 & 46872 \\
  $\frac{1+\sqrt{163}i}{2}$&8688 &640320&40133016\\
  \bottomrule
  
\end{tabular}
\end{center}

\section{reciprocal sums of Fibonacci numbers}\label{section_fibonacci}
\subsection{Generalized Fibonacci zeta function}
Let $\alpha,\beta$ be two algebraic numbers with $0<|\beta|<1$ and $\alpha\beta=\pm1$. We consider the Lucas sequence 
\[
U_n=\frac{\alpha^n-\beta^n}{\alpha-\beta}
\]
and its associated sequence
\[
V_n=\alpha^n+\beta^n.
\]
For example, if $\alpha=\frac{1+\sqrt{5}}{2}$, $\beta=\frac{1-\sqrt{5}}{2}$, then $U_n$ is the Fibonacci number $F_n$ and $V_n$ is the Lucas number $L_n$.

For $s> 0$ and integer $p\geq 0$, we define the generalized Fibonacci zeta function as follows:
\[
\zeta^p_{U}(s)=\sum_{n\geq 1}\frac{n^p}{U_n^s} \text{  and  }\zeta_{V}^p(s)=\sum_{n\geq 1}\frac{n^p}{V_n^s}.
\]
By the exponential growth property of $U_n$ and $V_n$, it is easy to show that the above series are convergent. The function $\zeta^0_F(s)$ was studied in \cite{MR1866357}, especially its meromorphic continuation. Similar methods can also be used to get the meromorphic continuation of $\zeta^p_{U}(s)$ and $\zeta^p_{V}(s)$. In this section, our main focus is on the special values of the above series at positive integers. For example, if $\alpha\beta=-1$ and $s$ is a positive integer, we see 
\[
\begin{aligned}
(\alpha-\beta)^{-s}\sum_{n\geq 1}\frac{n^p}{U_n^s}&=\sum_{n\geq 1}\frac{n^p}{((-\beta)^{-n}-\beta^n)^s}\\
&=\sum_{n\geq 1}\frac{(2n)^p}{(\beta^{-2n}-\beta^{2n})^s}+(-1)^s\sum_{n\geq 1}\frac{(2n-1)^p}{(\beta^{-2n+1}+\beta^{2n-1})^s}\\
&=2^p\sum_{n\geq 1}\frac{n^p\beta^{2ns}}{(1-\beta^{4n})^s}+(-1)^s\sum_{n\geq 1}\frac{(2n-1)^p\beta^{(2n-1)s}}{(1+\beta^{4n-2})^s}\\
&=2^p\operatorname{\Romannum{1}}_s^p(\beta^2)+(-1)^s\operatorname{\Romannum{4}}^p_s(\beta^2).
\end{aligned}
\]
Hence the special values of the generalized Fibonacci zeta function at positive integers coincide with the series introduced in Section \ref{section_intro}. We summarize these results in the lemma below.
\begin{lemma}\label{lemma_zetaFiboa}
Let $s>0$ and $p\geq 0$ be two integers. If $\alpha \beta =1$, then 
\[
\zeta_U^p(s)=(\alpha-\beta)^s\operatorname{\Romannum{1}}_s^p(\beta)\text{ and }\zeta_V^p(s)=\operatorname{\Romannum{2}}_s^p(\beta).
\]
If $\alpha \beta =-1$, then 
\[
\zeta^p_U(s)=(\alpha-\beta)^s\left(2^p\operatorname{\Romannum{1}}_s^p(\beta^2)+(-1)^s\operatorname{\Romannum{4}}^p_s(\beta^2)\right)\text{ and }\zeta^p_V(s)=2^p\operatorname{\Romannum{2}}^p_s(\beta^2)+(-1)^s\operatorname{\Romannum{3}}_s^p(\beta^2).
\]
\end{lemma}
 
Combining with the list provided in Theorem \ref{theorem_main}, we immediately obtain the following result.
\begin{corollary}\label{coro_zeta}
Keep the assumptions in Lemma \ref{lemma_zetaFiboa}. If $\alpha \beta =1$, then the numbers $\zeta_U^{2p}(2s)$, $\zeta_U^{2p+1}(2s+1)$, $\zeta_V^{2p}(2s)$, $\zeta_V^{2p}(2s+1)$ are special values of linear combinations of derivatives of Eisenstein series. If $\alpha \beta =-1$, then the numbers $\zeta_U^{2p}(2s)$,  $\zeta_V^{2p}(2s)$ are special values of linear combinations of derivatives of Eisenstein series.
\end{corollary}
We may also consider the series like
\[
\sum_{n\geq1}\frac{(-1)^{n-1}n^p}{U_n^s} \text{ or }\sum_{n\geq1}\frac{(2n-1)^p}{U_{2n-1}^s}.
\]
These series can also be expressed as linear combinations of derivatives of modular forms in some conditions. In this paper, we focus only on the series mentioned in Corollary \ref{coro_zeta}. 

\subsection{Results on algebraic independence}
Expressing the generalized reciprocal sums of the Lucas sequence as finite sums of derivatives of Eisenstein series allows us to explore their transcendence or algebraic independence. To begin with, we require several preliminary results. For algebraic independence of values of linear combinations of quasimodular forms, we have the following results.
\begin{theorem}\label{the_algind}\cite[Theorem 3.5]{MR4648485}
Let the functions $f$, $g$ and $h$ be algebraic linear combinations of quasimodular forms with algebraic Fourier coefficients. If they are algebraically independent, then for $z$ on the upper half plane such that $e^{2\pi iz}$ is an algebraic number, the three numbers $f(z)$, $g(z)$ and $h(z)$ are algebraically independent over $\mathbb{Q}$.
\end{theorem}

The above theorem immediately implies that the series mentioned in Corollary \ref{coro_zeta} are transcendental numbers. To show their algebraic independence, based on the above theorem, we need to decide the algebraic independence of the corresponding functions. Here are some criteria for the algebraic independence of linear combinations of quasimodular forms. 
\begin{lemma}\label{lemma_quasialg}
For each $i>0$ and a quasimodular form $f$, we denote it by $f_i$ to indicate that it is of weight $i$. Let $f_i,g_i,h_i$ be quasimodular forms. If the three functions $f=\sum_{i=1}^uf_i,g=\sum_{i=1}^vg_i,h=\sum_{i=1}^wh_i$ are algebraically dependent over $\mathbb{C}$, then  $f_u$, $g_v$, $h_w$ are algebraically dependent over $\mathbb{C}$.
\end{lemma}
\begin{proof}
The lemma follows easily by the fact that the ring of quasimodular forms is a graded ring with respect to the weights. Thus the quasimodular forms with different weights are linear independent. If there is a non-trivial algebraic relation $\sum c_{\alpha,\beta,\gamma}f^{\alpha}g^{\beta}h^{\gamma}=0$, let $d$ be the maximum of $\alpha u+\beta v+\gamma w$, then by homogeneity, $\sum_{\alpha u+\beta v+\gamma w=d} c_{\alpha,\beta,\gamma}x^{\alpha}y^{\beta}z^{\gamma}$ is a non-trivial algebraic relation among $f_u,g_v,h_w$.
\end{proof}

The above lemma illustrates that when considering the algebraic independence of linear combinations of quasimodular forms, we only need to consider those with the highest weights. Our series involve higher-order derivatives of modular forms, we also need the following criterion regarding the algebraic independence of higher-order derivatives of modular forms.
\begin{proposition}\label{prop_algindder}\cite[Proposition 4.2]{MR4648485}
Let $u,v,w$ be three positive integers and $a,b,c$ be three non-negative integers. Let $f,g,h$ be three modular forms of weights $u,v,w$ respectively. Suppose that the function $f^{bw-cv}g^{cu-aw}h^{av-bu}$ is not a constant. Then the three functions $\D^af,\D^bg,\D^ch$ are algebraically independent over $\mathbb{C}$.
\end{proposition}

For the case of $p=0$, we need the following lemma.
\begin{lemma}\label{lemma_modalg}
For each $i>0$ and a modular form $f$, we denote it by $f_i$ to indicate that it is of weight $i$. Let $f_i,g_i,h_i$ be modular forms.  If $f_u^vg_v^{-u}$ is not a constant, then the functions $f=\sum_{i=1}^uf_i,g=\sum_{i=1}^vg_i,h=E_2+\sum_{i=1}^wh_i$ are algebraically independent over $\mathbb{C}$. 
\end{lemma}
\begin{proof}
This follows from the algebraic independence of $E_2$ and any modular forms. More precisely, the identity $f_{N-2n}(z)E_2^n(z)+\cdots+f_{N}(z)=0$ implies that each modular form $f_{i}(z)$ is equal to zero. We can express the algebraic relation of $f,g,h$ as a polynomial in terms of $E_2$ and the coefficients are all modular forms. The above discussion indicates that the coefficient in front of the highest degree of $E_2$ is zero in this algebraic relation. This leads to an algebraic relation depending only on $f$ and $g$. When examining the highest degree term in this algebraic relation, we see that $f_u$ and $g_v$ are algebraically dependent. As modular forms, $f_u$ and $g_v$ are algebraically dependent if and only if $f_u^vg_v^{-u}$ is a constant (Cf. \cite[Lemma 3.1]{MR4648485}).
\end{proof}

\subsection{Examples}
For two complex numbers $x$ and $y$, the notation $ x\sim y$ means that $x=a y+b$ for two algebraic numbers $a$ and $b$, $a\neq 0$. When discussing algebraic independence, the roles played by $x$ and $y$ are identical. Utilizing the above results, we can obtain some conclusions regarding the algebraic independence of generalized reciprocal sums of the Lucas sequence. Firstly, by employing Lemma \ref{lemma_modalg}, the algebraic independence for $\zeta^0_U(s)$ and $\zeta^0_V(s)$ can be established. In the case of $\alpha\beta=-1$, the algebraic independence results for $\zeta^0_U(s)$ were proved by Elsner, Shimomura and Shiokawa in \cite{MR2794928}. 
\begin{theorem}
Let $s_1<s_2<s_3$ be three positive integers.
If $\alpha\beta=1$, the three numbers $\zeta_U^0(2s_1)$, $\zeta_U^0(2s_2)$, $\zeta_U^0(2s_3)$ are algebraically independent over $\mathbb{Q}$; The three numbers $\zeta_V^0(s_1),\zeta_V^0(s_2),\zeta_V^0(s_3)$ are algebraically dependent over $\mathbb{Q}$ if and only if at least one of $s_1,s_2,s_3$ is even.

If $\alpha\beta=-1$, the three numbers $\zeta_U^0(2s_1)$, $\zeta_U^0(2s_2)$,$\zeta_U^0(2s_3)$ are algebraically independent over $\mathbb{Q}$ if and only if at least one of $s_1,s_2,s_3$ is even; The three numbers $\zeta_V^0(2s_1)$,$\zeta_V^0(2s_2)$,$\zeta_V^0(2s_3)$  are algebraic independent over $\mathbb{Q}$ if and only if at least one of $s_1,s_2,s_3$ is even.
\end{theorem}
\begin{proof}
When $\alpha\beta=1$, let $z_0$ be an point on the upper half plane such that $\beta=e^{2\pi iz_0}$, we have
\[
\begin{aligned}
\zeta_U^0(2s)&=(\alpha-\beta)^{2s}\operatorname{\Romannum{1}}_{2s}^0(\beta)=(\alpha-\beta)^{2s}\sum_{r=1}^{s}t(2s,2r) A_{2r-1}(\beta^2)\\
&=(\alpha-\beta)^{2s}\sum_{r=1}^{s}t(2s,2r) \frac{B_{2r}}{4r}(1-E_{2r}(2z_0))\\
&\sim E_{2s}(2z_0)+\cdots+a_s E_2(2z_0),
\end{aligned}
\]
where $a_s$ is a non-zero constant. By Theorem \ref{the_algind}, to show the algebraic independence of the numbers $\zeta_U^0(2s_1)$, $\zeta_U^0(2s_2)$, $\zeta_U^0(2s_3)$,  we are reduced to proving the  algebraic independence of the functions $f_s:=E_{2s}(2z_0)+\cdots+a_s E_2(2z_0)$, $s=s_1,s_2,s_3$. To use Lemma \ref{lemma_modalg}, we instead focus on the algebraic independence of $\left\{a_{s_3}^{-1}f_{s_3}-a_{s_1}^{-1}f_{s_1},a_{s_2}^{-1}f_{s_2}-a_{s_1}^{-1}f_{s_1},f_{s_1}\right\}$. Using Lemma \ref{lemma_modalg}, it is sufficient to show that the function $E_{2s_3}^{-s_2}E_{2s_2}^{s_3}$ is not a constant. Considering its Fourier coefficients, it suffices to show $B_{2s_3}\neq B_{2s_2}$. By $(-1)^{n-1}B_{2n}=2(2n)!(2\pi)^{-2n}\zeta(2n)$, we have the estimate
\begin{equation}\label{equ_Ber}
\frac{2(2n)!}{(2\pi)^{2n}}<|B_{2n}|<\frac{2(2n)!}{(2\pi)^{2n}(1-2^{1-2n})},
\end{equation}
hence $|B_{2s_3}|>|B_{2s_2}|$ holds with the only exception $B_4=B_8=-1/30$, for more details see \cite[Lemma 5]{MR2593248}. We need to check $f_4$, $f_2$ and $f_1$ more carefully and the algebraic independence of these three functions follows by the algebraic independence of $E_2$, $E_4$ and $E_6$. This finishes the proof of the algebraic independence of $\zeta_U^0(2s_1)$, $\zeta_U^0(2s_2)$, $\zeta_U^0(2s_3)$. We also have for some non-zero constants $b_s$ and $c_s$,
\[
\begin{aligned}
&\zeta_V^0(2s)\sim E_{2s}(2z_0)-4^{s}E_{2s}(4z_0)+\cdots+b_s\left( E_{2}(2z_0)-2E_{2}(4z_0)\right)-2b_s E_2(4z_0)\\
&\zeta_V^0(2s+1)\sim E_{2s+1}^{1,\chi_{-4}}(z_0)+\cdots+c_s E_1^{1,\chi_{-4}}(z_0).
\end{aligned}
\]
From the above expressions, we immediately obtain that the terms with odd indices and those with even indices are algebraically independent. Furthermore, since the terms with odd indices only involve modular forms, by the fact that any three modular forms are algebraically dependent, we obtain that any three terms with odd indices are algebraically dependent. For the algebraic independence among the terms with even indices, using a similar argument as before, it suffices to show that the function
\[
\left(E_{2s_3}(z)-4^{s_3}E_{2s_3}(2z)\right)^{-s_2}\left(E_{2s_2}(z)-4^{s_2}E_{2s_2}(2z)\right)^{s_3}
\]
is not a constant. Considering its Fourier coefficients, it suffices to show that $(4^{s_3}-1)B_{2s_3}\neq (4^{s_2}-1)B_{2s_2}$ and this follows easily by the inequality (\ref{equ_Ber}). By 
\[
B_{2k+1,\chi_{-4}}=(-1)^{k-1}(2k+1)!L(\chi_{-4},2k+1)\left(\frac{2}{\pi}\right)^{2k+1}\text{ and } \frac{2}{3}< L(\chi_{-4},k)<1,
\]
we have the estimate
\[
\frac{2^{2n+2}(2n+1)!}{3\pi^{2n+1}}<|B_{2n+1},\chi_{-4}|<\frac{2^{2n+1}(2n+1)!}{\pi^{2n+1}}.
\]
This inequality implies that $|B_{2i+1,\chi_{-4}}|>|B_{2j+1,\chi_{-4}}|$ for $i>j\geq 0$, and then the functions $E_{2s+1}^{1,\chi_{-4}}(z)$ for different $s$ are mutually algebraically independent.
This proves the algebraic independence among the terms with even indices.

When $\alpha\beta=-1$, let $z_0$ be an point on the upper half plane such that $\beta^2=e^{2\pi iz_0}$, we have
\[
\begin{aligned}
\zeta_U^0(2s)\sim &4^sE_{2s}(4z_0)-(4^s+1+(-1)^s)E_{2s}(2z_0)+E_{2s}(z_0)+\cdots \\
&+d_s\left(4E_{2}(4z_0)-(5+(-1)^s)E_{2}(2z_0)+E_{2}(z_0)\right).
\end{aligned}
\]
for some non-zero constants $d_s$. If $s_1,s_2,s_3$ are odd, the above expressions involve only modular forms, thus any three terms with odd indices are algebraically dependent. If at least one of $s_1,s_2,s_3$ is even, we need to  to consider cases based on the parity of $s_1,s_2,s_3$. 
If only one number is even, for example $s_1$, then we can use Lemma \ref{lemma_modalg}, and we are reduced to showing that the functions 
$g_s=4^sE_{2s}(4z)-4^sE_{2s}(2z)+E_{2s}(z) ~,s=s_3,s_2,$
are algebraically independent, i.e. $g_{s_2}^{s_3}g_{s_2}^{-s_3}$ is not a constant. Considering its Fourier coefficients, it suffices to show $B_{2s_3}\neq B_{2s_2}$, and this has already been proved in the above statements. If two numbers are even, for example $s_1$ and $s_2$, then we consider the algebraic independence of  $\left\{\zeta^0_U(2s_3),\zeta^0_U(2s_2)-a\zeta^0_U(2s_1),\zeta^0_U(2s_1)\right\}$ instead, where the constant $a$ is chosen so that the function corresponding to $\zeta^0_U(2s_2)-a\zeta^0_U(2s_1)$ is a sum of some modular forms. At this point, we can once again use Lemma \ref{lemma_modalg}, and we need to show that the two functions $g_{s_3}$ and $4^{s_2}E_{2s_2}(4z)-(4^{s_2}+2)E_{2s_2}(2z)+E_{2s_2}(z)$ are algebraically independent. Considering their Fourier series, it suffices to show that $B_{2s_3}\neq -B_{2s_2}$, and this has already been proven. The same method can be applied to the remaining cases, which we omit here. We also have
\[
\begin{aligned}
\zeta_V^0(2s)\sim& 4^sE_{2s}(4z)-(1+(-1)^s)E_{2s}(2z)+(-1)^sE_{2s}(z)+\cdots\\
&+e_s\left(4E_{2}(4z)-(1+(-1)^s)E_{2}(2z)+(-1)^sE_{2}(z)\right).
\end{aligned}
\]
The case of $\zeta_V^0(2s)$ is essentially the same as before. We leave the details to the readers.
\end{proof}

In the case of $p>0$, when using Proposition \ref{prop_algindder} to decide the algebraic independence, the following theorem may be useful.
\begin{theorem}[\cite{MR4177529}]\label{theo_eisalgind}
The equation
\[
\prod_{i=1}^nE_{k_i}=\prod_{j=1}^mE_{l_j}
\]
with $k_i\neq l_j$ for any $1\leq i\leq n$, $1\leq j\leq m$ has no solutions, except for those identities given by
\[
E_4^2=E_8,~E_4E_6=E_{10},~E_4E_{10}=E_{14},~E_6E_8=E_{14}.
\]
\end{theorem}	
The above theorem can show the non-existence of certain equalities about Eisenstein series. We can use this theorem to demonstrate the algebraic independence of several quasimodular forms.
\begin{proposition}\label{theo_alg_zetau}
Let $s_1,s_2,s_3$ be three distinct positive integers and $p_1,p_2,p_3$ be three distinct non-negative integers. Suppose that $s_i\geq p_i+4$ or $p_i\geq s_i+3$ and $p_is_j\neq p_js_i$, $(2s_i-1)(p_j-s_j+1)\neq (2s_j-1)(p_i-s_i+1)$, $(2s_i-1)(s_j-p_j)\neq 2p_j(p_i-s_i+1)$ for distinct $i,j\in\{1,2,3\}$. If $\alpha\beta=1$, then the three numbers $\zeta^{2p_i}_U(2s_i)$, $i=1,2,3$ are algebraically independent over $\mathbb{Q}$.
\end{proposition}
\begin{proof}
We have the expression
\[
\begin{aligned}
\zeta_U^{2p}(2s)&=(\alpha-\beta)^k\text{I}_{2s}^{2p}(\beta)\\
&=\sum_{r=1}^{p}t(2s,2r)\D^{2r-1}A_{2p-2r+1}(\beta^2)+\sum_{r=p+1}^{s}t(2s,2r)\D^{2p} A_{2r-2p-1}(\beta^2),
\end{aligned}
\]
and by Theorem \ref{the_algind}, we only need to focus on the highest degree terms, which is
\[
\D^{2s-1}E_{2p-2s+2}(z)\text{ if } p\geq s \text{, and }\D^{2p}E_{2s-2p}(z)\text{ if }1\leq p\leq s-1.
\]
We can use Proposition \ref{prop_algindder} to determine the algebraic independence of the above functions. And this reduces to determining certain identities about Eisenstein series. The inequality conditions prevent trivial cases. Conditions $s_i\geq p_i+4$ or $p_i\geq s_i+3$ eliminates the exceptions in Theorem \ref{theo_eisalgind}. Under these conditions, Theorem \ref{theo_eisalgind} implies that there is no equation about Eisenstein series, and this completes the proof.
\end{proof}

Using the above method, we can obtain more results on algebraic independence. We give an example to illustrate that determining the algebraic independence of certain reciprocal sums of Fibonacci numbers can be reduced to some numerical computations.
\begin{proposition}
Let $1\leq s_1<s_2<s_3<300$ be three integers. If $\alpha\beta=-1$, then the three numbers $\zeta^2_U(2s_1)$, $\zeta^2_U(2s_2)$, $\zeta^2_U(2s_3)$ are algebraically independent over $\mathbb{Q}$.
\end{proposition}
\begin{proof}
When discussing the algebraic independence of quasimodular forms, by Lemma \ref{lemma_quasialg}, we only need to consider the highest weight terms. After some calculations, we can see the functions correspond to highest weight terms of the numbers $\zeta_U^2(2)$ and $\zeta_U^2(2s)$ for $s\geq 2$ are $\D f_1$ and $\D^2 f_s$, respectively, where
\[
f_s(z)=\left\{\begin{aligned}
&4E_2(4z)-2E_2(2z)+E_2(z) &s=1,\\
&4^{s-1}E_{2s-2}(4z)-(4^{s-1}+1+(-1)^s)E_{2s-2}(2z)+E_{2s-2}(z) &s\geq 2.
\end{aligned}\right.
\]
We need to consider two cases separately. To begin with, assume that $1=s_1<s_2<s_3$. By Proposition \ref{prop_algindder}, to prove that the functions $\D f_1,\D^2 f_{s_2},\D^2 f_{s_3}$ are algebraically independent, it suffices to show that $f_1^{2s_3-2s_2}f_2^{3-s_3}f_3^{s_2-3}$ is not a constant. Considering its Fourier expansion, it suffices to show 
\[
8(s_3-s_2)+\frac{(2s_2-2)(3-s_3)}{(-1)^{s_2-1}B_{2s_2-2}}+\frac{(2s_3-2)(s_2-3)}{(-1)^{s_3-1}B_{2s_3-2}}\neq 0.
\]
Numerical computations show that the above inequalities hold for $s_3<300$. If $1<s_1<s_2<s_3$, again by Proposition \ref{prop_algindder}, it is sufficient to show that the function $f_{s_1}^{s_3-s_2}f_{s_2}^{s_1-s_3}f_{s_3}^{s_2-s_1}$ is not a constant. Considering its Fourier expansion, the proof is completed by showing  
\[
\frac{(s_3-s_2)(s_1-1)}{(-1)^{s_1}B_{2s_1-2}}+\frac{(s_1-s_3)(s_2-1)}{(-1)^{s_2}B_{2s_2-2}}+\frac{(s_2-s_1)(s_3-1)}{(-1)^{s_3}B_{2s_3-2}}\neq 0.
\]
Numerical computations demonstrate that the above inequalities hold for $s_3<300$ with only one exception $s_1=4,s_2=6,s_3=8$. A more carefully calculation shows that the function $f_4^2f_6^{-4}f_8^2=1 - 230400q^2+O(q^3)$ is not a constant, which completes the proof.
\end{proof}

\providecommand{\bysame}{\leavevmode\hbox to3em{\hrulefill}\thinspace}
\providecommand{\MR}{\relax\ifhmode\unskip\space\fi MR }
\providecommand{\MRhref}[2]{%
  \href{http://www.ams.org/mathscinet-getitem?mr=#1}{#2}
}
\providecommand{\href}[2]{#2}

\end{document}